\RequirePackage{fix-cm}
\documentclass{article}
\usepackage{hyperref}

\usepackage{listings}
\usepackage{extarrows}
\usepackage{lipsum}     
\usepackage{geometry}   
\usepackage{titlesec}   
\usepackage{epstopdf}
\usepackage{pgfplots}
\usepackage{scalefnt}
\usepackage{tikz}
\usetikzlibrary{shadows}
\usepackage{geometry}
\usetikzlibrary{patterns}

\usepackage{booktabs}
\usepackage{diagbox} 
\usepackage{multirow} 
\usepackage{makecell} 
\usepackage{longtable} 

\usepackage{graphicx}
\usepackage{epstopdf}

\usepackage{amsthm}
\usepackage{algorithm} 
\usepackage{algorithmic} 
\usepackage{CJKutf8}
\usepackage{enumerate}
\usepackage{pifont}
\usepackage{indentfirst}
\usepackage{framed}
\usepackage{amssymb}
\usepackage{mathrsfs}
\usepackage{graphicx}
\usepackage{graphicx}
\usepackage{amsmath}
\usepackage{subeqnarray}
\usepackage{cases}
\usepackage{caption}
\usepackage{flushend,cuted}
\usepackage{color}
\usepackage{float}
\usepackage{subfigure}
\usepackage{multirow}
\usepackage{appendix}
\newcommand{\be}{\begin{equation}}
\newcommand{\ee}{\end{equation}}
\newcommand{\bit}{\begin{itemize}}
\newcommand{\eit}{\end{itemize}}
\def\bfx{{\bf x}}
\def\bfa{{\bf a}}

\def\bfy{{\bf y}}
\def\face{{\rm face}}
\def\diag{{\rm diag}}
\def\Diag{{\rm Diag}}

\def\la{\langle}
\def\ra{\rangle}

\def\rank{{\rm rank}}
\newtheorem{proposition}{Proposition}[section]
\newtheorem{theorem}[proposition]{Theorem}
\newtheorem{lemma}[proposition]{Lemma}

\newtheorem{definition}[proposition]{Definition}
\newtheorem{remark}[proposition]{Remark}
\newtheorem{example}[proposition]{Example}

\newcommand{\mtB}{\mathcal{B}}

\newcommand{\mS}{\mathcal{S}}
\newcommand{\mK}{\mathcal{K}}

\begin{document}

\title{A Facial Reduction Approach for the Single Source Localization Problem}
\author{He Shi\thanks{School of Mathematics and Statistics, Beijing Institute of Technology, Beijing, 100081, P. R. China}
\and Qingna Li\thanks{Corresponding author. School of Mathematics and Statistics/Beijing Key Laboratory on MCAACI/Key Laboratory of Mathematical Theory and Computation in Information Security, Beijing Institute of Technology, Beijing, 100081, P. R. China. Email: qnl@bit.edu.cn. This author’s research is supported by the National Science Foundation of China (NSFC) 12071032.}}

\maketitle

\abstract{The single source localization problem (SSLP) appears in several fields such as signal processing and global
positioning systems. The optimization problem of SSLP is nonconvex {and} difficult to find {its globally optimal}
solution. It can be reformulated as a rank constrained Euclidean distance matrix (EDM) completion problem with a number of equality constraints. In this paper, we propose a facial reduction approach to solve {such an EDM} completion problem. For the
constraints of fixed distances between sensors, we reduce them to a face of the EDM cone and derive the closed
formulation of the face. We prove constraint nondegeneracy for each feasible point of the resulting EDM optimization problem
{without a rank constraint, which guarantees the quadratic convergence of semismooth Newton's method}. To tackle the nonconvex rank constraint, we apply the majorized penalty approach developed by
Zhou et al. (IEEE Trans Signal Process 66(3):4331-4346, 2018). Numerical results verify the fast speed of the proposed
approach while giving comparable quality of solutions as other methods.}




\maketitle

\section{Introduction}
The single source localization problem (SSLP) has received extensive {attentions} such as in emergency rescue \cite{emergency}, monitoring \cite{monitor}, tracking and navigating \cite{navigating}, the ad hoc microphone array \cite{micro1, micro2, echo}, etc. The objective is to locate a source that is detected by a set of sensors with noisy observations between the source and the sensors. More specifically, given locations $\bfx_1,...,\bfx_n\in \mathbb{R}^r$ of $n$ sensors (usually $r=2$ or $3$ for visualization) and noisy observation ${\delta_{j,n+1}} >0$ between the $j$-th sensor and unknown source $\mathbf{x}_{n+1} \in \mathbb{R}^r$, where
\be\label{noisy}
\delta_{j,n+1}=\|\mathbf{x}_j-\mathbf{x}_{n+1}\|_2+\epsilon_j,\ j=1,...,n,
\ee
and $\epsilon_j$ is some random noise, we are looking for source ${\mathbf{x}_{n+1}} \in \mathbb{R}^r$ such that $\|\mathbf{x}_j-{\mathbf{x}_{n+1}}\|_2$ matches ${\delta_{j,n+1}},\ j=1,\cdots,n $ as much as possible.

To seek this approximation, typical vector based models can be summarized into the following least squares model
\be\label{LS}
\min_{\bfx_{n+1} \in\mathbb{R}^r}\sum_{j=1}^n(\|\mathbf{x}_j-\mathbf{x}_{n+1}\|_2^p-\delta_{j,n+1}^p)^2
\ee 
with $p=1$ or $2$. Problem (\ref{LS}) is nonconvex, and is in particular nonsmooth in the case of $p=1$. Different approaches have been developed including the simple fixed point algorithm (SFP), the sequential weighted least squares algorithm (SWLS) \cite{ex2}{, the} constrained weighted least squares approach (CWLS) \cite{Cheung} for $p=1$, { and the} least square approach based on squared-range measurements (SR-LS) for $p=2$ \cite{ex1} in which problem (\ref{LS}) leads to the generalized trust region subproblem (GTRS) that can be solved efficiently and globally.


As a special case of sensor network localization (SNL), SSLP has also been frequently studied in the community of semidefinite programming (SDP) \cite{SNLSDP, SDP1, Henry_shu, So2007, Vaghefi}.
Let the Gram matrix be defined by $Y:=P^{\top}P$, with $P=[\bfx_1,...,\bfx_{n+1}]\in\mathbb{R}^{r\times (n+1)}$.
By the famous linear transformation $\mathcal{K}: \mathcal{S}^{n+1}\rightarrow\mathcal{S}^{n+1}$ between the positive semidefinite matrices (PSD) cone and the EDM cone \cite{Kmap} defined by
$$
\mathcal{K}(Y)_{ij}=Y_{ii} +Y_{jj} -2Y_{ij},
$$
SSLP can be formulated as an SDP with a number of linear constraints
\begin{equation}\label{sdpsslp}
  \begin{split}
    \min_{Y\in\mathcal{S}^{n+1}} & \ \frac{1}{2}\|\mathcal{K}(Y)-\Delta\|_F^2 \\
     \mathrm{s.t.} &\ \rank(Y)\le r\\
      &\ Ye_{n+1}=0\\
      &\ Y\in \mathcal{S}^{n+1}_{+}\\
      &\ \mathcal {K}(Y)_{ij} = \|\bfx_i-\bfx_j\|_2^2,\ 1\le i<j\le n,
  \end{split}
\end{equation}
where $\Delta$ is the given estimated squared distance matrix, $e_{n+1}=(1,...,1)^{\top}\in\mathbb{R}^{n+1}$, $\mathcal{S}^{n+1}$ denotes the space of $(n+1)\times (n+1)$ real symmetric matrices and $\mathcal{S}^{n+1}_{+}$ denotes the cone of PSD. Such SDP can be solved by the well-developed SDP packages such as SeDuMi \cite{sedumi}, SDPT3 \cite{SDPT3} and the recent SDPNAL+ \cite{SDPNAL+} or some specially designed algorithms such as SNLSDP \cite{SNLSDP} and SFSDP \cite{SFSDP}.

{A separated line} of investigating SSLP is from the Euclidean distance matrix (EDM) optimization point of view. The success of EDM {models} to tackle the distance based optimization problems has been demonstrated in \cite{semi, EMBED, qiLAG} to deal with multidimensional scaling as well as {the} related problems. Notice that squared distance {matrix} $D\in\mathcal{S}^{n+1}$ defined by
\be\label{Dmat}
D_{ij}:=\|\bfx_i-\bfx_j\|_2^2,\ i,j=1,\cdots ,n+1
\ee
is {an EDM.} Here, $ r$ is the embedding dimension of $D$. In \cite{qiLAG}, Qi et al. proposed a Lagrangian dual method to solve the EDM model (which is equivalent to (\ref{LS}) with $p=2$):
\begin{align}\label{nedm}
  \min_{D\in\mathcal{S}^{n+1}} & \ \frac{1}{2}\|D-\Delta\|_F^2 \\
  \text{s.t.} & \ -D\in \mathcal{K}^{n+1}_+,\ \diag(D)=0\tag{\ref{nedm}{a}}\label{nedm-1}\\
    &\ \rank({J_{n+1} DJ_{n+1}})\le r\tag{\ref{nedm}{b}}\label{nedm-2}\\
    &\ D_{ij}=\|\bfx_i-\bfx_j\|_2^2,\ 1\leq i < j \leq n,\tag{\ref{nedm}{c}}\label{nedm-3}
\end{align}
where $\diag(D)$ is the vector formed by the diagonal elements of $D$,
$\mathcal{K}_+^{n+1} :=\{D\in\mathcal{S}^{n+1}\mid v^{T}Dv\ge 0,\ \sum_{i=1}^{n+1} v_i=0\}$, ${J_{n+1}} $ is the centralization matrix defined as {$J_{n+1} :=I_{n+1}-\frac{1}{n+1}e_{n+1} e_{n+1}^{\top}$} with identity matrix $I_{n+1}\in\mathbb{R}^{(n+1)\times (n+1)}$. The
basic difference between (\ref{sdpsslp}) and (\ref{nedm}) is that in (\ref{nedm}), one uses (\ref{nedm-1}) directly to build up the model, rather than using the Gram matrix via $\mathcal{K}$.

Notice that in SDP model (\ref{sdpsslp}) and EDM model (\ref{nedm}), there are a lot of equality constraints, describing the fixed distances between sensors. The number of these equality constraints is in the order of $O(n^2)$, which grows fast as $n$ grows. {Therefore, how to tackle these equality constraints not only becomes the main concern for SSLP, but also has its own interest in solving models like (\ref{sdpsslp}) and (\ref{nedm}) more efficiently.}
Very recently, the facial reduction technique was introduced to tackle the equality constraints in SDP. Specifically, Krislock and Wolkowicz \cite{Henry_shu} {used} facial reduction to solve a semi-definite relaxation of SNL, which has been {proven} to be much faster than the traditional SDP method. In \cite{psd}, Sremac et al. proposed to reformulate SDP model (\ref{sdpsslp}) as the following SDP problem based on the facial reduction technique
\begin{equation}\label{henrysslp}
  \begin{split}
    \min_{Y\in\mathcal{S}^{n+1}} & \ \frac{1}{2}\|\mathcal{K}(Y)-\Delta\|_F^2 \\
     \mathrm{s.t.} &\ \rank(Y)\le r\\
      &\ Y\in\mathrm{face}({\Theta},\mathcal{S}^{n+1}_+).
  \end{split}
\end{equation}
In (\ref{henrysslp}), the known distances between sensors are described in the set of constraints
$${\Theta}:=\{Y\in \mathcal{S}^{n+1}_{+}\mid\ Ye_{n+1}=0,\ \mathcal {K}(Y)_{ij} = \Delta_{ij},\ 1\le i,j\le n\},$$
and the Gram matrix $Y$ is restricted to the minimal face of PSD {(denoted as face($\Theta , \mathcal{S}^{n+1}_+$))} through facial reduction. {Roughly speaking, the minimal face containing a subset is a face that does not contain any other face which contains this subset (See Definition \ref{minface} for the formal definition).}
Then (\ref{henrysslp}) reduces to the following SDP model in a smaller scale:
\begin{equation}\label{new1}
\begin{split}
  \min_{R\in\mathcal{S}^{r+1}} & \ \frac{1}{2}\|\mathcal{K}(CRC^T)-\Delta\|_F^2 \\
   \mathrm{s.t.} &\ \rank(R)\le r\\
    &\ R\in\mathcal{S}^{r+1}_+,
\end{split}
\end{equation}
where $C\in\mathbb{R}^{(n+1)\times(r+1)}$ is a constant matrix.
The above two references \cite{qiLAG, psd} bring our attentions to the facial reduction technique which we will briefly review.

The idea of facial reduction is to take the intersection of a cone and a set of linear constraints as the constraint on some face of the cone. By characterizing the face property, we can study the corresponding optimization problem restricted on the face of the cone.
The faces of cones were investigated since 1980's \cite{Barker1981, Henry1981, Henry1981-2}, whereas the research on the face of positive semidefinite (PSD) cone was studied by Hill et al. \cite{Hill1987} and the general formulation of PSD faces was derived therein. Very recently, facial reduction techniques have been used in algorithms for several important scenarios including principal component analysis (PCA) \cite{HenryPCA} and matrix optimization problems \cite{Henry_shu, psd14}.
In contrast, the study on faces of the EDM cone is mainly from the theoretical point of view. {The EDM cone is the set of EDMs, defined by
$$
\mathcal{E}^{n}=\left\{D\in\mathcal{S}^n\mid \exists\ x_1,\cdots,x_n\in \mathbb{R}^r, \ \mathrm{such}\ \mathrm{that}\ D_{ij}=\|x_i-x_j\|^2_2,\ i,j=1,\cdots,n\right\}.
$$
We refer to \cite{JDJ, Kn+} for other characterizations of the EDM cone. }
In \cite{Tarazaga1996}, Tarazaga et al. showed that the faces of the EDM cone are linear mappings of that of the PSD cone, which provides a way to study the faces of the EDM cone. In \cite{Tarazaga}, Tarazaga described the faces of the EDM cone and related these faces to the supporting hyperplane of the EDM cone. In order to further obtain a more explicit expression of the EDM cone, it {was} proved by Alfakih \cite{Alfakih_remark} that faces of the EDM cone is a Gale subspace related to EDM. In his later monograph \cite{Alfakih_M}, he derived the minimal face of the EDM cone containing a matrix.

Based on the above analysis, EDM model (\ref{nedm}) for SSLP has its advantages over other models {like (\ref{sdpsslp}). Meanwhile} the facial reduction technique reduces a large scale problem (\ref{sdpsslp}) to a smaller one in (\ref{new1}), which is also attractive. Therefore, a natural question is whether one can apply the facial reduction technique to EDM model (\ref{nedm}). This motivates the work in our paper. Based on EDM model (\ref{nedm}) for SSLP, we develop a facial reduction model as follows 
\begin{equation}\label{nedmf}
\begin{split}
  \min_{D\in\mathcal{S}^{n+1}} &\ \frac{1}{2}\|{D}-\Delta\|_F^2 \\
  \text{s.t.} &\  \mathrm{rank}({J_{n+1} DJ_{n+1}})\leq r \\
    &\ {D} \in \mathrm{face}({\Omega},\mathcal{E}^{n+1}),
\end{split}
\end{equation}
where
\be\label{ft}
{\Omega }:=\left\{D\in \mathcal{E}^{n+1}\mid\ D_{ij}={\|\bfx_i-\bfx_j\|_2^2},\ 1\leq i,j \leq n\right\},
\ee
which is equivalent to (\ref{nedm}). We show that $\mathrm{face}({\Omega},\mathcal{E}^{n+1})$ can be further characterized by a simple linear constraint $\la D,H\ra=0$, which brings an equivalent facial reduction of (\ref{nedmf}) and we call it the EDM model based on facial reduction (EDMFR):
\begin{align}\label{EDMFR_former}
  \min_{D\in \mathcal{S}^{n+1}} &\ \frac{1}{2}\|D -\Delta\|_F^2 \\
  \text{s.t.} &\ -D\in\mathcal{K}^{n+1}_+ ,\ \diag(D)=0\tag{\ref{EDMFR_former}{a}}\label{EDMFR-1}\\
    &\ {\rank({J_{n+1} DJ_{n+1}})\le r}\tag{\ref{EDMFR_former}{b}}\label{EDMFR-2}\\
    &\ \la D,H\ra=0.\tag{\ref{EDMFR_former}{c}}\label{EDMFR-3}
\end{align}
Compared with (\ref{nedm}), the $n(n-1)/2$ linear equality constraints in (\ref{nedm-3}) is replaced by one linear equality constraint (\ref{EDMFR-3}), which significantly reduces the number of equality constraints.

The contributions of our paper {are} in three folds. Firstly, we derive model (\ref{EDMFR_former}) for SSLP based on facial reduction. The advantage of EDM model (\ref{EDMFR_former}) is that it greatly reduces the extra number of equality constraints from $n(n-1)/2$ to one. Secondly, we show constraint nondegeneracy of the convex case of model (\ref{EDMFR_former}), where the rank constraint is dropped, {which makes it possible to solve the corresponding convex problem by the semismooth Newton's method in \cite{semi}.} Thirdly, inspired by the work in \cite{qi2018}, we design a fast algorithm called majorized penalty approach to solve model (\ref{EDMFR_former}) and verify the competitive performance of the algorithm by various numerical results.

The organization of the paper is as follows. In Section \ref{sec2}, we show how to derive EDMFR (\ref{EDMFR_former}) from (\ref{nedmf}). We prove constraint nondegeneracy for the convex relaxation problem of EDMFR in Section \ref{secCN}. In Section \ref{sec3}, we show how to apply the majorized penalty approach to solve (\ref{EDMFR_former}). In Section \ref{sec4}, we report the numerical results to demonstrate the efficiency of our approach. We conclude the paper in Section \ref{sec5}.

\emph{Notations.} The interior of $\mathcal{S}^n_+$ is denoted by $\mathcal{S}^n_{++}$. Let vec: $\mathcal{S}^n \rightarrow \mathbb{R}^{n(n+1)/2}_+$ map
the upper triangular elements of a symmetric matrix to a vector.
For $x\in\mathbb{R}^n$, let Diag($x$) be the diagonal matrix consisting of vector $x$. We denote the null space of a matrix $X$ by null$(X)$. Finally, we use $A_{ij}$ to denote the $(i,j)$ element in matrix $A$.



\section{Characterization of EDM Face ${\mathrm{face}(\Omega,\mathcal{E}^{n+1})}$}\label{sec2}
In this part, we will show how to derive the explicit form of $\face (\Omega,\mathcal{E}^{n+1})$. We need {the following} definitions including face, exposed face and minimal face, which can be found in Definition 1.1, 1.2 and Theorem 1.27 in \cite{Alfakih_M}.

Let $C$ be a convex {set} in an Euclidean space $\mathbb{E}$. A convex subset $F\subseteq C$ is a {\it face} of $C$, if for any $x,y\in C$ and any point $z$ lying between $x$ and $y$, one has $z\in F$ implies that $ x,y\in F$.

{The definitions of exposed face and minimal face are given as follows.}
\begin{definition}\label{expface}
{Let $C$ be a convex set in an Euclidean space $\mathbb{E}$. A face $F$ of $C$ is an {\it exposed face} when satisfying one of the following conditions:}
\bit
  \item[(i)] {There exists a supporting hyperplane $\mathcal{H}$ such that $F=C\cap \mathcal{H}$.}
  \item[(ii)] {There exists a vector $v$ satisfying $F=C\cap v^{\perp}$, where $v^{\bot}$ is defined by $v^{\bot}:=\{x\in\mathbb{E}\mid \la v,x\ra=0\}.$ In this case, we say that $v$ exposes $F$ and $v$ is the {\it exposing vector} of $F$.}
\eit

\noindent{If $C$ is a convex cone, condition (ii) can be reduced to the following condition.}
\bit
  \item[(iii)]
       {There exists a vector $v\in C^*$ satisfying $F=C\cap v^{\perp}$,} where $C^*$ is the dual cone of $C$ defined by
$$
C^*:=\left\{x^*{\in\mathbb{E}}\ \mid\ \la x,x^*\ra\ge 0,\ \forall\ x\in C\right\}.
$$
\eit
\end{definition}

\begin{definition}{\rm\cite{psd14}}\label{minface}
Let $C$ be a convex set in an Euclidean space $\mathbb{E}$. The {\it minimal face} containing a set $S\subseteq C,$ denoted as $\face(S,C)$, is the intersection of all faces of $C$ containing $S$.
\end{definition}
{
Below we give some simple examples to illustrate the above definitions.}

\begin{example}\label{face1}
{As shown in Fig. \ref{facesex}, let $C=\{x\in \mathbb{R}^3\mid 0\le x_1,\ x_2,\ x_3\le 2\}$. $C_1=\{x\in \mathbb{R}^3\mid x_3=0,\ 0\le x_1,\ x_2\le 2\}$ and $C_2=\{x\in \mathbb{R}^3\mid x_2=2,\ 0\le x_1,\ x_3\le 2\}$ are faces of $C$. It can be easily shown that $C_1$ and $C_2$ are exposed faces of $C$. 
For example, for $C_1$, we can find the supporting hyperplane $\mathcal{H}=\{x\in\mathbb{R}^3\mid x_3=0,\ x_1,x_2\in\mathbb{R}\}$ such that $C_1=C\cap \mathcal{H}$. Alternatively, by choosing $v=(0,0,1)^{\top}$, there is $C_1=C\cap v^{\bot}$. Therefore, the exposing vector of $C_1$ is $v=(0,0,1)^{\top}$.

The convex subset $S=\{x\in \mathbb{R}^3\mid x_2=2,\ (x_1-1)^2+(x_3-1)^2\le\frac{1}{4}\}$ is not a face of $C$. However, we can find the minimal face containing $S$, which is $C_2$, i.e.,  face$(S,C)=C_2$. It can thus be seen that the minimal face can reduce the feasible region from $C$ to a face of $C$, which achieve the purpose of dimension reduction.}
\begin{figure}[H]
  \centering
  \includegraphics[width=0.5\textwidth]{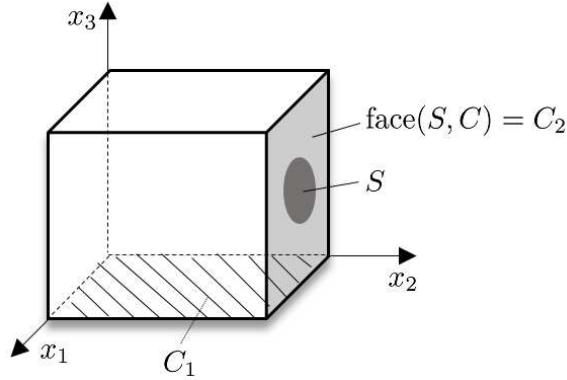}\\
  \caption{Demonstration of face, minimal face in Example \ref{face1}}\label{facesex}
\end{figure}
\end{example}
{It should be pointed out that not all faces are exposed. We give a simple example.}
\begin{example}\label{face2}
{Consider the convex set $C\subseteq\mathbb{R}^2$ described in Fig. \ref{unexpose}, where
$$C=\left\{x\in\mathbb{R}^2\mid x_1^2+x_2^2\le 1\right\}\bigcup \left\{x\in\mathbb{R}^2\mid -1\le x_1\le 1,\ -1\le x_2\le 0\right\}.$$
Consider the points $y=(-1,0)^{\top}$ and $z=(1,0)^{\top}$. One can verify that $\{y\}$ and $\{z\}$ are faces of $C$. $\{y\}$ is not an exposed face of $C$. The reason is that there is no hyperplane $\mathcal{H}$, such that $\{y\}=C\cap \mathcal{H}$. Similarly, $\{z\}$ is not an exposed face of $C$.}
\begin{figure}[H]
  \centering
  \includegraphics[width=0.5\textwidth]{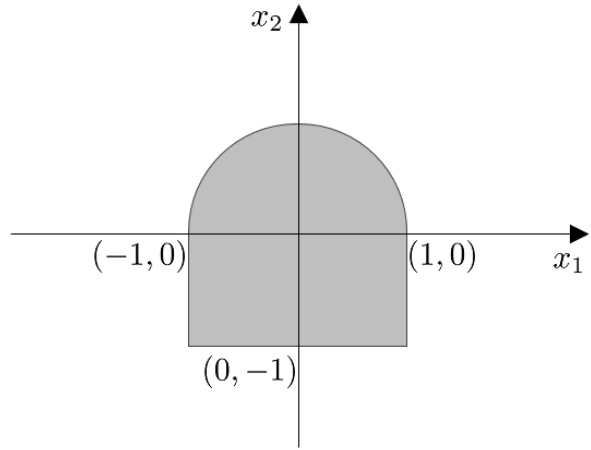}\\
  \caption{Demonstration of Example \ref{face2}.}\label{unexpose}
\end{figure}
\end{example}

Having introduced the definitions related to face above, we are ready to derive the explicit form of $\face(\Omega, \mathcal{E}^{n+1})$. {Here $\Omega$ is defined as in (\ref{ft}).} 
Since the faces of the EDM cone are exposed \cite[Section 2.1]{psd14}, the minimal face $\face(\Omega,\mathcal{E}^{n+1})$ is exposed as well. As a result, the only thing we need to do is to find the exposing vector of $\face(\Omega,\mathcal{E}^{n+1})$. To that end, we first derive the exposing vector of $\face(\overline{D},\mathcal{E}^{n})$, where $\overline{D}\in\mathcal{S}^n$ is defined by
\be\label{Dbar}
\overline{D}_{ij}=\|\bfx_i-\bfx_j\|_2^2,\ i,j=1,...,n.
\ee
Then inspired by \cite{psd}, we extend the result to a higher dimensional case, and derive the exposing vector for $\face(\Omega,\mathcal{E}^{n+1})$.
Hence, the derivation of $\face(\Omega,\mathcal{E}^{n+1})$ is divided into the following two steps. 

\textbf{(1) Exposing vector of $\face(\overline{D},\mathcal{E}^{n})$.}

To derive the exposing vector of $\face(\overline{D},\mathcal{E}^{n})$, we use the tool of projected Gram matrix. The {\it projected Gram matrix} $Y$ of $\overline{D}$ is defined by
\be\label{proGram}
Y:=-\frac{1}{2}{V ^{\top} \overline{D} V}, \ \text{where} \
    {V}=\left[\begin{array}{c}
    -\frac{1}{\sqrt{n}}e_{n-1}^{\top}\\
    I_{n-1}-\frac{1}{n+\sqrt{n}}e_{n-1}e_{n-1}^{\top}
    \end{array} \right]\in\mathbb{R}^{n\times (n-1)}.
\ee
{It is easy to verify that} $V$ satisfies $V^{\top} e_n=0 ,\ {V^{\top} V}=I_{n-1} $ and {${VV^{\top} }=J_n$.} 

{
There are many theories about faces of the EDM cone that have been reflected in \cite{Alfakih_M}. Here we briefly describe the results. The following two definitions are important to characterize the associated hyperplane of $\face(\overline{D},\mathcal{E}^{n})$.
\begin{definition}{\rm(}Gale subspace{\rm)}
  Let $\widetilde{D}\in \mathcal{E}^n $ be a given {\rm EDM}, $P=[p^1\cdots p^n]^{\top}\in \mathbb{R}^{n\times r} $ be the coordinate matrix, and
$$
\left[
  \begin{array}{c}
    P^{\top} \\
    e^{\top} \\
  \end{array}
\right]=\left[
          \begin{array}{ccc}
            p^1 & ... & p^n \\
            1 & ... & 1 \\
          \end{array}
        \right]
\in \mathbb{R}^{(r+1)\times n}.$$
Then the corresponding \emph{Gale subspace} is
$$
G(\widetilde{D}):=\mathrm{null}\left(\left[
  \begin{array}{c}
    P^{\top} \\
    e^{\top} \\
  \end{array}
\right]\right)
=\mathrm{null}(P^{\top})\cap \mathrm{null}(e^{\top}).$$
\end{definition}
\begin{definition}{\rm(}Gale matrix{\rm)}
If $G(\widetilde{D})$ denotes the Gale subspace of an EDM $\widetilde{D}$, then a Gale matrix of $\widetilde{D}$ is a matrix whose columns form the basis for $G(\widetilde{D})$.
\end{definition}
{The following lemma relates $\text{face}(\widetilde{D} ,\mathcal{E}^n)$ to a supporting hyperplane of the EDM cone, which plays an important role in our following analysis.}
\begin{lemma}\label{hyperplane}{\rm\cite[{\it Theorem} 5.8]{Alfakih_M}}
 Let $F$ be a proper face of $\mathcal{E}^n$. Let $\widetilde{D}\in\mathrm{relint}(F)$ and $Z_1$ be a Gale matrix of $\widetilde{D}$. Let
$\mathcal{H}=\{D\in \mathcal{E}^n: \ \mathrm{trace}(Z_1 ^{\top} DZ_1)=0\} $, then $F=\mathcal{E}^n\cap \mathcal{H} $. {Here $\mathrm{relint}(F)$ denotes the relative interior of $F$.}
\end{lemma}

\begin{lemma}{\rm\cite[{\it Theorem} 11.6]{Rockafellar}}\label{Rock}
Let $C$ be a convex set and $S$ be the nonempty convex set of $C$. In order that there exists a supporting hyperplane $\mathcal{H}$ to $C$ containing $S$ and $C\nsubseteq \mathcal{H}$, if and only if $S\cap \mathrm{relint}(C)=\emptyset$.
\end{lemma}

{Lemma \ref{Rock} leads to the following result.}

\begin{lemma}\label{relint}
Let $\overline{D}$ be given by {\rm(\ref{Dbar})} and $F$ be the minimal face of $\mathcal{E}^n$ containing $\overline{D}$. That is $F=\mathrm{face}(\overline{D},\mathcal{E}^n)$. Then $\overline{D}\in \mathrm{relint}(F)$.
\end{lemma}

\begin{proof}

For contradiction, if $\overline{D}\notin\mathrm{relint}({F})$, there is $\{\overline{D}\}\cap\mathrm{relint}(F)=\emptyset$. Since $\overline{D}\in F$, $\overline{D}$ is on the relative boundary of $F$. By Lemma \ref{Rock}, there exists a supporting hyperplane (denoted as $\mathcal{H}_1$) to $F$ and $F\nsubseteq \mathcal{H}_1$. Let $F'=F\cap \mathcal{H}_1$. Consequently, \begin{equation}\label{eq-1}
F'\subseteq F, \ F'\neq F.
\end{equation}
Now for $F'$, we have found a supporting hyperplane $\mathcal{H}_1$ such that $F'=F\cap \mathcal{H}_1$. Therefore, by Definition \ref{expface} (i), $F'$ is an exposed face of $F$.
Recall that $F$ is a face of $\mathcal{E}^n$. Together with the fact that faces are transitive \cite[Theorem 1.26]{Alfakih_M}, we have that $F'$ is also a face of $\mathcal{E}^n$. Note that $\overline{D}\in F$ and $\overline{D}\in \mathcal{H}_1$, there is $\overline{D}\in F'$. In other words, $F'$ is a face of $\mathcal{E}^n$ containing $\overline{D}$.
By Definition \ref{minface}, $F$ is the intersection of all faces of $\mathcal{E}^n$ containing $\overline{D}$. Therefore, $F\subseteq F'$. This contradicts with (\ref{eq-1}). Therefore, $\overline{D}\in\mathrm{relint}({F})$.

\end{proof}

{Based on Lemma \ref{hyperplane}, Lemma \ref{Rock} and Lemma \ref{relint}, we can give the explicit form of the exposing vector for face($\overline{D} ,\mathcal{E}^n $). }
}

\begin{proposition}\label{exp-vector}
Let $\overline{D}$ be given by {\rm(\ref{Dbar})} and $Y$ be the projected Gram matrix defined by {\rm(\ref{proGram})}. Define $Z_1\in\mathbb{R}^{n\times(n-r-1)}$ by
    \be\label{galmat}
    Z_1 :=V U,
    \ee
where $U\in \mathbb{R}^{(n-1)\times(n-r-1)}$ is a matrix whose columns form orthonormal bases of {\rm null}$(Y)$.\footnote{By \cite[Lemma 3.8]{Alfakih_M}, $Z_1$ is a Gale matrix.} Then $-Z_1 Z_1 ^{\top}  $ exposes {\rm face($\overline{D} ,\mathcal{E}^n $)}.
\end{proposition}
\begin{proof}
Let $F:=\face(\overline{D} ,\mathcal{E}^n )$. {By Lemma \ref{relint},} we have $\overline{D}\in\mathrm{relint}({F})$. Therefore,
by {Lemma \ref{hyperplane}}, we can find a hyperplane in which face $F$ containing matrix $\overline{D}$ lies. Hence, by the above analysis and {Lemma \ref{hyperplane}}, we have {\rm face($\overline{D} ,\mathcal{E}^n $)}$=\mathcal{E}^n\cap \mathcal{H} $, where $\mathcal{H}=\{D_1\in \mathcal{E}^n: \ \mathrm{trace}(Z_1 ^{\top} D_1 Z_1)=0\} $. In other words,
$$
-Z_1 Z_1 ^{\top} \in \mathcal{H}^{\bot}.
$$
Next we will prove $-Z_1 Z_1 ^{\top} \in (\mathcal{E}^{n})^{*}$. 
Indeed, we have
$$
\begin{array}{ll}
  (\mathcal{E}^{n})^{*} & =\{A \in \mathcal{S}^n\ \mid \ \langle A, B\rangle\geq 0, \ \forall\ B\in \mathcal{E}^n\} \\
   & =\{A \in \mathcal{S}^n\ \mid \ \langle A, B\rangle\geq 0,\ \forall\ B\in \mathcal{K}(\mathcal{S}^n_+)\}\ \ \ \\
    & =\{A \in \mathcal{S}^n\ \mid \ \langle A, \mathcal{K}(Y)\rangle\geq 0,\ \forall \ Y\in \mathcal{S}^n _+\}\\
    & =\{A \in \mathcal{S}^n\ \mid \ \langle \mathcal{K}^{*}(A), Y\rangle\geq 0,\ \forall\ Y\in \mathcal{S}^n _+\}\\
    & =\left\{A\in \mathcal{S}^n\ \mid \ \mathcal{K}^* (A)\in \mathcal{S}^n _+ \right\} \\
     & = \left\{A\in \mathcal{S}^n\ \mid \ \mathrm{Diag}(A {e_{n} } )-A\in \mathcal{S}^n _+ \right\},
\end{array}
$$
{where the second equality is due to the fact that $ \mathcal{K}(\mathcal{S}^n_+)=\mathcal{E}^n,$
as in \cite[Page 1163, Line 20]{psd14} and \cite[Page 975, Line 37]{psd}.} Moreover, there is $\mathrm{Diag}(-Z_1 Z_1 ^{\top} e_n )-(-Z_1 Z_1 ^{\top})=Z_1 Z_1 ^{\top}$ (Because $Z_1 \in \mathrm{null}(e_{n}^{\top})$). It is obvious that $Z_1 Z_1 ^{\top}\in \mathcal{S}^n _+$. Consequently, $-Z_1 Z_1 ^{\top}\in (\mathcal{E}^{n})^{*}$.

By the definition of {the exposing vector}, $-Z_1 Z_1 ^{\top}  $ exposes face($\overline{D} ,\mathcal{E}^n $). The proof is completed.
\end{proof}


\textbf{(2) Exposing vector of $\face(\Omega,\mathcal{E}^{n+1})$.}



In order to {derive} the exposing vector of $\mathrm{face}(\Omega ,\mathcal{E}^{n+1})$, as in \cite[Section 2.3]{psd14}, we consider an undirected graph $G = (V_g, E) $ with vertex set $V_g = \{1, \ \ldots, \ n, n+1 \}$ and edge set $E = \{ij: \ 1\le  i \le j \leq n \} $. Define projection map $\mathcal{P}$ : ${\mathcal{S}^{m}\rightarrow \mathbb{R}^{n(n+1)/2}}$ by
$$
\mathcal{P}(A)=(A_{ij})_{ij\in E},
$$
i.e., $\mathcal {P}(A)$ is a vector of all entries of $A$ indexed by E. The adjoint of $\mathcal{P}$ is $\mathcal{P}^{*}:\mathbb{R}^{n(n+1)/2}\rightarrow \mathcal{S}^{n+1}$ defined by
$$(\mathcal{P}^{*}(y))_{ij}=\left\{\begin{array}{ll}
y_{ij}/2, & \mathrm{i}\mathrm{f}\ i,j\in E\ \mathrm{or}\ {j,i}\in E,\\
0, & \mathrm{o}\mathrm{t}\mathrm{h}\mathrm{e}\mathrm{r}\mathrm{w}\mathrm{i}\mathrm{s}\mathrm{e}.
\end{array}\right.$$

We have the following result to characterize the exposing vector of $\mathrm{face}(\Omega ,\mathcal{E}^{n+1})  $.
\begin{theorem}\label{face}
Let $v:=\mathrm{vec}(-Z_1 Z_1 ^{\top})\in \mathbb{R}^{n(n+1)/2}$, where $Z_1$ is defined by {\rm(\ref{galmat})}. Let $H=\mathcal{P}^{*}(v)$, then $H$ exposes $\mathrm{face}(\Omega ,\mathcal{E}^{n+1})  $.
\end{theorem}

\begin{proof}
Let $a=\mathrm{vec}(\overline{D})\in \mathbb{R}^{n(n+1)/2}_+$. By \cite[Lemma 4.11]{psd14}, we have
$$
-Z_1 Z_1 ^{\top}{/2} \ \mathrm{exposes} \ \mathrm{face}(\overline{D},\mathcal{E}^{n})\ \Leftrightarrow v \ \mathrm{exposes} \ {\mathrm{face}(a,\mathrm{vec}(\mathcal{E}^n))}.
$$
{Moreover, we have the following formulation
$$
\mathrm{vec}(\mathcal{E}^n)=\mathcal{P}(\mathcal{E}^{n+1}).
$$
} The feasible domain defined by formula (\ref{ft}) can be rewritten as
$$
F =\{D\in \mathcal{E}^{n+1}\mid \ \mathcal{P}(D)=a\}.
$$
Then by \cite[Theorem 4.1]{psd14}, we have
$$v \ \mathrm{exposes} \ \mathrm{face}(a,{\mathcal{P}(\mathcal{E}^{n+1})})\Leftrightarrow \ \mathcal{P}^{*}(v)\ \mathrm{exposes} \ \mathrm{face}(\Omega,\mathcal{E}^{n+1}),$$
where
$$\mathcal{P}^{*}(v)=\left[
                                      \begin{array}{cc}
                                        -Z_1 Z_1 ^{\top}/2 & 0 \\
                                        0 & 0 \\
                                      \end{array}
                                    \right].
$$
In other words, $H=\mathcal{P}^{*}(v)$ exposes $\mathrm{face}(\Omega ,\mathcal{E}^{n+1})  $. This completes the proof.
\end{proof}

Through $H$, we {can} perform "null space" representation of face($\Omega,\mathcal{E}^{n+1}$). In other words,
\be\label{face-eq}
  \face(F,\mathcal{E}^{n+1})=\mathcal{E}^{n+1}\cap H^{\bot}.
\ee
where $H^{\bot}$ is given by
\be\label{expH}
H^{\bot}:=\{D\in\mathcal{S}^{n+1}\mid\ \la D,H\ra=0\}.
\ee
With model (\ref{nedmf}), (\ref{face-eq}) and (\ref{expH}), we obtain the equivalent problem of model (\ref{nedmf})
\begin{equation}\label{EDMFR}
\begin{split}
  \min_{D\in \mathcal{S}^{n+1}} &\ \frac{1}{2}\|D -\Delta\|_F^2 {=:f(D)}\\
  \text{s.t.} &\ \mathrm{rank}({J_{n+1} DJ_{n+1}})\leq r \\
    &\ D\in \mathcal{E}^{n+1}\cap H^{\bot},
\end{split}
\end{equation}
which is exactly EDMFR (\ref{EDMFR_former}). Compare (\ref{EDMFR_former}) with (\ref{nedm}), we recast the $n(n-1)/2$ equality constraints in (\ref{nedm}c) by a simple linear constraint $\la D,H\ra=0$, which means that we consider the optimization problem restricted on $\face(\Omega,\mathcal{E}^{n+1})$. Here we summarize the calculation of exposing vector $H$ in Algorithm \ref{alg-exp} followed by a toy example.

\begin{algorithm}
\caption{Calculate Exposing Vector $H$}\label{alg-exp}
\hspace*{0.02in} {\bf Input:} 
$\bfx_1,...,\bfx_n\in\mathbb{R}^r$.
\begin{algorithmic}[1]

\STATE Calculate $\overline{D}\in\mathcal{S}^n$ by $\overline{D}=\|\bfx_i-\bfx_j\|_2^2,\ i,j=1,...,n$.
\STATE Calculate $Y$ by (\ref{proGram}).
\STATE Calculate $Z_1$ by (\ref{galmat}).
\STATE Let $H= \left[
                                             \begin{array}{cc}
                                               -Z_1Z_1^{\top}/2 & 0 \\
                                               0 & 0 \\
                                             \end{array}
                                           \right]$.
\end{algorithmic}
\hspace*{0.02in} {\bf Output:} 
$H\in\mathcal{S}^{n+1}$.
\end{algorithm}


\begin{example}\label{XH=0}
The data comes from {\rm\cite[Example 1]{ex1}}. We consider the case where $n =5$, i.e., there are five sensor points. We have known sensors points $\mathbf{x}_i$, distributed in two dimensions, whose coordinate matrix is defined by
$$
\begin{array}{rl}
  P&:=\left[\mathbf{x}_1 \ \mathbf{x}_2 \ \mathbf{x}_3 \ \mathbf{x}_4 \ \mathbf{x}_5\right] \\
    & \ =\left[
         \begin{array}{ccccc}
           6 & 0 & 5 & 1 & 3 \\
           4 & -10 & -3 & -4 & -3 \\
         \end{array}
       \right].
\end{array}
$$
The true coordinate of the source is $\mathbf{x}_6=(-2,3)^{\top}$. 
To calculate exposing vector $H$, we need the following steps.
\bit
\item[1.] Calculate $\overline{D}$ by $\overline{D}=\|\bfx_i-\bfx_j\|_2^2$. That is:
    $$
    \overline{D}=\left[
    \begin{array}{ccccc}
            0 & 232 &  50 &  89 &  58  \\
  232    &    0  & 74 & 37 & 58\\
   50& 74&        0  & 17&    4   \\
   89&   37& 17&        0  &  5  \\
   58&   58&    4&  5&         0
    \end{array}
  \right].
  $$
\item[2.] Calculate the projected Gram matrix of $\overline{D}$ by
 $$
 Y=-V^{\top}\overline{D}V/2=\left[
              \begin{array}{cccc}
             6.8760&	0.2266&	1.8570&	0.6432\\
             0.2266&	2.1796&	0.4886&	1.1029\\
             1.8570&	0.4886&	2.9208&	1.3554\\
             0.6432&	1.1029&	1.3554&	2.0261
              \end{array}
            \right],
 $$
 where $V$ is calculated by {\rm(\ref{proGram})}.
\item[3.] 
Calculate $U$ by
$$
Y=:W\Lambda W^{\top}.
$$
Then $U\in\mathbb{R}^{4\times 2}$ is the columns in $W$ corresponding to the zero eigenvalues. That is,
$$
U=\left[
    \begin{array}{cc}
0.2145&	0.0648\\
0.7041&	-0.4482\\
-0.6768&	-0.4313\\
0.0126&	0.7803
    \end{array}
  \right].
$$
Calculate $Z_1$ by {\rm(\ref{galmat})}:
$$
Z_1=VU=\left[
         \begin{array}{cc}
-0.1138&	0.0153\\
0.1794&	0.0696\\
0.6689&	-0.4435\\
-0.7120&	-0.4265\\
-0.0226&	0.7851
         \end{array}
       \right].
$$
\item[4.] {Exposing vector} $H$ is given by
$$
H=\left[
    \begin{array}{cc}
      -Z_1Z_1^{\top}/2 & 0 \\
      0 & 0 \\
    \end{array}
  \right]
=\left[
                \begin{array}{cccccc}
-0.0066&	0.0097&	0.0415&	-0.0372&	-0.0073&	0\\
0.0097&	-0.0185&	-0.0446&	0.0787&	-0.0253&	0\\
0.0415&	-0.0446&	-0.3221&	0.1436&	0.1816&	0\\
-0.0372&	0.0787&	0.1436&	-0.3444&	0.1594&	0\\
-0.0073&	-0.0253&	0.1816&	0.1594&	-0.3084&	0\\
0&	0&	0	&0&	0	&0
                \end{array}
              \right].
$$
\eit
Next, we will show that EDM matrix $D\in\mathcal{S}^6$ between sensors $\bfx_1,...,\bfx_5$ and source $\bfx_6$ is indeed lying on $\face(\Omega,\mathcal{E}^6)$. In fact, $D$ is given by
 $$
 D=\left[
    \begin{array}{cccccc}
            0 & 232 &  50 &  89 &  58 &65 \\
  232    &    0  & 74 & 37 & 58&173\\
   50& 74&        0  & 17&    4 &85  \\
   89&   37& 17&        0  &  5 &58 \\
   58&   58&    4&  5&         0&61\\
   65 & {173} & {85} & {58} & {61}&0
    \end{array}
  \right]\in\mathcal{S}^6
 $$
and $\la D,H\ra=-1.7764\times 10^{-14}\approx 0$. Therefore, $H$ is indeed the exposing vector of $\face(\Omega,\mathcal{E}^6)$.
\end{example}

\begin{remark}
To the best of our knowledge, we are the first to use facial reduction technique in EDM model. Based on {\rm(\ref{EDMFR_former})}, we solve {the} EDM model for SSLP on a face of the EDM cone, greatly reducing the number of constraints. 
\end{remark}

\section{Constraint Nondegeneracy for Convex Case}\label{secCN}
Note that EDMFR (\ref{EDMFR_former}) is a nonconvex optimization problem due to the rank constraint in (\ref{EDMFR-2}). A popular way to deal with the nonconvexity is to simply drop the rank constraint and one will reach a convex EDM model as follows
\begin{equation}\label{convex}
\begin{split}
  \min_{D\in \mathcal{S}^{n+1}} &\ \frac{1}{2}\|D -\Delta\|_F^2 \\
  \text{s.t.} &\ {-D\in\mathcal{K}^{n+1}_+} \\
    &\ \mathcal{B}(D)=0.
\end{split}
\end{equation}
Here linear operator  $\mtB:\ \mathcal{S}^{n+1}\rightarrow\ \mathbb{R}^{n+2}$ is defined by
$$
\mtB(D):=\left[
           \begin{array}{c}
             {\diag(D)} \\
              \la {D},H\ra\\
           \end{array}
         \right]
$$
with adjoint operator $\mtB^*:\ \mathbb{R}^{n+2}\rightarrow\ \mathcal{S}^{n+1}$ defined by $\mtB^*(y):=\Diag(y_{1:n+1})+y_{n+2}H$.

Problem (\ref{convex}) is essentially in the same form as {the} EDM model in \cite{semi}, where $\mathcal{B}(D)$ is replaced by $\diag(D)$. Therefore, to solve (\ref{convex}), one can also apply the semismooth Newton's method, as done in \cite{semi}. However, to guarantee the quadratic convergence and the nonsingularity of the generalized Jacobian of the dual problem, one needs constraint nondegeneracy property, which is an essential property behind the good performance of semismooth Newton's method. Therefore, in this section, we will show that constraint nondegeneracy indeed holds for the constraints in (\ref{convex}). Below we first state the definition of constraint nondegeneracy with respect to the constraints in (\ref{convex}). More details about this property for general constraints can be found in \cite{CN, sunCN}.

\begin{definition}{\rm\cite[Definition 3.2]{CN}}
 Constraint nondegeneracy holds at a feasible point $D$ with respect to the constraints in (\ref{convex}), if the following holds
\begin{equation}\label{jianbing}
\mathcal{B}\left(\mathrm{lin}(\mathcal{T}_{\mathcal{K}_+^{n+1}}(D))\right)=\mathbb{R}^{n+2},
\end{equation}
where $\mathrm{lin}(\mathcal{T}_{\mathcal{K}_+^{n+1}}(D))$ denotes the largest linear subspace contained in the tangent cone of {$\mathcal{K}^{n+1}_+$} at $D$.
\end{definition}


To show constraint nondegeneracy, we need the following property.
\begin{proposition}\label{prop2}
For $A\in \mathcal{S}^{n+1},\ {{\bf a}\in\mathbb{R}^n,}\ c\in\mathbb{R}$ and
{$$
B:=\left[
     \begin{array}{cc}
       0 & {\bf a} \\
       {{\bfa}}^{\top} & {c} \\
     \end{array}
   \right]\in\mathcal{S}^{n+1},
$$}
we have
$$
\mathrm{trace}(AB)={u_{n+1}^{\top}}A\left[
          \begin{array}{c}
            2{{\bfa}}\\
            {c}
          \end{array}
        \right],
$$
{where $u_{n+1} $ represent the $(n+1)$-th column of $(n+1)\times (n+1)$ identity matrix $I_{n+1}$.}
\end{proposition}
\begin{proof}
With simple calculations, there is
\begin{equation*}
\begin{split}
                     \mathrm{trace}(AB)& = \sum_{i,j}A_{ij}B_{ij} \\
                      & = \sum_{i=1}^{n}A_{i,n+1}B_{i,n+1}+ \sum_{j=1}^{n}A_{n+1,j}B_{n+1,j}+A_{n+1,n+1}{c}\\
                      & = \sum_{j=1}^{n}A_{n+1,j}(2{{{\bfa}}_j})+A_{n+1,n+1}{c}\\
                      & = {u_{n+1}^{\top}}A\left[
          \begin{array}{c}
            2{\bfa} \\
            {c}
          \end{array}
        \right].
\end{split}
\end{equation*}
This completes the proof.
\end{proof}

\begin{theorem}
Constraint nondegeneracy {\rm(\ref{jianbing})} holds for each feasible point $D$ of {\rm(\ref{convex})}.
\end{theorem}
\begin{proof}
We divide the proof into two steps.

In step one, we recall the explicit form of $\mathrm{lin}(\mathcal{T}_{\mathcal{K}_+^{n+1}}(D))$ whose result can be obtained from \cite{semi}.
Let $D\in\mathcal{K}_+^{n+1}$ have the following decomposition:
\begin{equation}\label{QAQ}
D:=Q\left[
                   \begin{array}{cc}
                     \overline{Z}_1 & {\bf\overline{z}}_2 \\
                     {\bf\overline{z}}_2^{\top} & \overline{z}_0 \\
                   \end{array}
                 \right]Q,\ \overline{Z}_1\in\mathcal{S}^{n},\ {\bf\overline{z}}_2\in\mathbb{R}^{n},\ \overline{z}_0\in\mathbb{R},
\end{equation}
where $Q\in\mathcal{S}^{n+1}$ is the Householder matrix satisfying $Q^2=I_{n+1}$:
$$
Q:=I_{n+1} - \frac{1}{{n+1}+\sqrt{n+1}}{\bfy\bfy}^{\top},\ {\bfy}=(1,...,1,\sqrt{n+1}+1)\in\mathbb{R}^{n+1}.
$$
Let $l:=\rank(\overline{Z}_1),\ \overline{\lambda}_1\ge\overline{\lambda}_2\ge\cdots\overline{\lambda}_l>0$ be the positive eigenvalues of $\overline{Z}_1$ and
\be\label{z1fenjie}
\overline{Z}_1=U\left[
                  \begin{array}{cc}
                    \overline{\Lambda} &   \\
                      & 0_{n-l} \\
                  \end{array}
                \right]U^{\top},\ U^{\top}U=I_l
\ee
{be} the spectral decomposition of $\overline{Z}_1$, {where $\overline{\Lambda}=\diag(\overline{\lambda}_1, \overline{\lambda}_2,\cdots,\overline{\lambda}_l)\in\mathcal{S}^{l}$.} Let $\overline{l}:=\{1,2,...,l\}$. It can be described by \cite[eq. (23)]{semi} that the tangent cone of $\mathcal{K}_+^{n+1}$ at $D$ is given as follows
$$
\mathcal{T}_{\mathcal{K}_+^{n+1}}(D):=\left\{Q\left[\begin{array}{cc}
                                                        U\left[
                                                            \begin{array}{cc}
                                                              \Sigma_1 & \Sigma_{12} \\
                                                              \Sigma_{12}^{\top} & \Sigma_2 \\
                                                            \end{array}
                                                          \right]U^{\top}
                                                          & {{\bfa}} \\
                                                         {\bfa}^{\top} & a_0
                                                       \end{array}\right]Q:\ \ \begin{array}{c}
                                                                          \Sigma_1\in\mathcal{S}^{l}, \ \Sigma_2\in\mathcal{S}^{n-l}_+\\
                                                                          \Sigma_{12}\in\mathbb{R}^{l\times(n-l)} \\
                                                                          {\bfa}\in\mathbb{R}^{n},\ a_0\in\mathbb{R}
                                                                        \end{array}\right\}.
$$
Then the largest linear subspace contained in $\mathcal{T}_{\mathcal{K}_+^{n+1}}(D)$ is given by \cite[eq. (24)]{semi}
\be\label{linT}
\mathrm{lin}(\mathcal{T}_{\mathcal{K}_+^{n+1}}(D)):=\left\{Q\left[\begin{array}{cc}
                                                        U\left[
                                                            \begin{array}{cc}
                                                              \Sigma_1 & \Sigma_{12} \\
                                                              \Sigma_{12}^{\top} & 0 \\
                                                            \end{array}
                                                          \right]U^{\top}
                                                          & {\bfa} \\
                                                         {\bfa}^{\top} & a_0
                                                       \end{array}\right]Q:\ \ \begin{array}{c}
                                                                          \Sigma_1\in\mathcal{S}^{l} \ \\
                                                                          \Sigma_{12}\in\mathbb{R}^{l\times(n-l)} \\
                                                                          {\bfa}\in\mathbb{R}^{n},\ a_0\in\mathbb{R}
                                                                        \end{array}\right\}.
\ee

{In step two, we will show (\ref{jianbing}) holds.}
By the characterization of $\mathrm{lin}(\mathcal{T}_{\mathcal{K}_+^{n+1}}(D))$ in (\ref{linT}), we can obtain that
$$A=Q\left[
          \begin{array}{cc}
            \left[
              \begin{array}{cc}
                kI_l & 0 \\
                0 & 0 \\
              \end{array}
            \right]
             & {\bfa} \\
            {\bfa}^{\top} & c \\
          \end{array}
        \right]Q\in \mathrm{lin}(\mathcal{T}_{\mathcal{K}_+^{n+1}}(D)),\ \ \forall \ [{\bfa}^{\top},c,k]^{\top}\in \mathbb{R}^{n+2}.
$$
Given arbitrary ${{\bf b}}\in\mathbb{R}^{n+2}$, to show that (\ref{jianbing}) holds, we want to find $[{\bfa}^{\top},c,k]^{\top}\in \mathbb{R}^{n+2}$ such that
\be\label{Aw=b}
\mtB(A)={\bf b}.
\ee
{Let $u_{i} $ represent the $i$-th column of $(n+1)\times (n+1)$ identity matrix $I_{n+1}$.} By \cite[Theorem 2.3]{semi}, we have
\begin{equation*}
\begin{split}
A_{ii}&=u_{i}^{\top}Q\left[
          \begin{array}{cc}
            \left[
              \begin{array}{cc}
                kI_l & 0 \\
                0 & 0 \\
              \end{array}
            \right]
             & {\bfa} \\
            {\bfa}^{\top} & c \\
          \end{array}
        \right]Q{u_{i}}\\
&=u_{i}^{\top}Q\left[
              \begin{array}{cc}
                kI_l & 0 \\
                0 & 0 \\
              \end{array}
            \right]
             Q{u_{i}}+({u_{i}})^{\top}Q\left[
          \begin{array}{cc}
           0  & {\bfa} \\
            {\bfa}^{\top} & c \\
          \end{array}
        \right]Q{u_{i}}\\
&=\left\{\begin{array}{cr}
          -\frac{1}{\sqrt{n+1}}u_{i}^{\top}Q\left[
          \begin{array}{c}
            2{\bfa} \\
            c
          \end{array}
        \right]+k,&\mathrm{if}\ i\le l, \\
          -\frac{1}{\sqrt{n+1}}{u_{i}}^{\top}Q\left[
          \begin{array}{c}
            2{\bfa} \\
            c
          \end{array}
        \right],&\mathrm{otherwise}.
         \end{array}\right.\ \ \ \ \qquad \text{(by\ \cite[Theorem \ 2.3]{semi})}
\end{split}
\end{equation*}
Hence,
\begin{equation*}
      \diag(A)=  -\frac{1}{\sqrt{n+1}}e_{n+1}^{\top}Q\left[
          \begin{array}{c}
            2{\bfa} \\
            c
          \end{array}
        \right]+\diag\left(\left[
              \begin{array}{cc}
                kI_l & 0 \\
                0 & 0 \\
              \end{array}
            \right]\right).
\end{equation*}
On the other hand, we have
\begin{equation*}
\begin{split}
\la A,H\ra&=\left\la Q\left[
          \begin{array}{cc}
           0  & {\bfa} \\
            {\bfa}^{\top} & c \\
          \end{array}
        \right]Q,H\right\ra+\left\la Q\left[
              \begin{array}{cc}
                kI_l & 0 \\
                0 & 0 \\
              \end{array}
            \right]
             Q,H\right\ra\\
 &=\mathrm{trace}\left(QHQ\left[
              \begin{array}{cc}
               0  & {\bfa} \\
            {\bfa}^{\top} & c \\
              \end{array}
            \right]\right)+\mathrm{trace}\left(QHQ\left[
              \begin{array}{cc}
                kI_l & 0 \\
                0 & 0 \\
              \end{array}
            \right]\right)\\
 &=u_{n+1}^{\top}(QHQ)\left[
          \begin{array}{c}
            2{\bfa} \\
            c
          \end{array}
        \right]+kl_q\ \qquad\qquad\quad (\mathrm{by} \ \mathrm{Proposition\ \ref{prop2}})\\
 &=-\frac{1}{\sqrt{n+1}}e_{n+1}^{\top}HQ\left[
          \begin{array}{c}
            2{\bfa} \\
            c
          \end{array}
        \right]+kl_q\qquad\quad (\mathrm{by} \ Q{u_{n+1}}=-\frac{1}{\sqrt{n+1}}{e_{n+1}})\\
 &=kl_q , \ \ \ \ \ \qquad\qquad\qquad\qquad\qquad\qquad\quad (\mathrm{by} \ {e_{n+1}^{\top}}H=0)
\end{split}
\end{equation*}
{where $l_q:=\sum_{i=1}^l(QHQ)_{ii}.$} Hence, we have
\begin{equation*}
\begin{split}
\mtB(A)
 &=\left[
           \begin{array}{c}
             \diag(A) \\
             \la A,H\ra\\
           \end{array}
         \right]\\
 &=\left[\begin{array}{c}
      -\frac{1}{\sqrt{n+1}}e_{n+1}^{\top}Q\left[
          \begin{array}{c}
            2{\bfa} \\
            c
          \end{array}
        \right]+kw\\
        kl_q
     \end{array}\right]\\
 &=M\left[
            \begin{array}{c}
              {\bfa} \\
              c \\
              k \\
            \end{array}
          \right],
\end{split}
\end{equation*}
where
$$
M:=\left[
     \begin{array}{cc}
       -\frac{1}{\sqrt{n+1}}Q & w \\
       0 & l_q \\
     \end{array}
   \right]\left[
            \begin{array}{cc}
              2I_n & 0 \\
              0 & I_2  \\
            \end{array}
          \right],\ w=\diag\left(\left[
              \begin{array}{cc}
                I_l & 0 \\
                0 & 0 \\
              \end{array}
            \right]\right)
   .
$$
Since $-H \succeq 0$, then $-QHQ\succeq 0$ with rank $n-r-1$. Then we have $\sum_{i=1}^l(QHQ)_{ii}\neq 0$. Hence, $M $ is invertible due to the fact that $Q$ is invertible and $l_q\neq 0$.

Hence, for any ${\mathbf{b}}\in\mathbb{R}^{n+2}$, we can find $[{\bfa}^{\top},c,k]^{\top}\in\mathbb{R}^{n+2}$ such that $[{\bfa}^{\top},c,k]^{\top}=M^{-1}{\mathbf{b}}$. For such ${\bfa}$, $c$ and $k$, we have ${\mathbf{b}}=\mtB(A)\in \mathcal{B}\left(\mathrm{lin}(\mathcal{T}_{\mathcal{K}_+^{n+1}}(D))\right)$. Thus (\ref{jianbing}) holds and hence constraint nondegeneracy holds at $D$. This completes the proof.
\end{proof}

{
Constraint nondegeneracy guarantees the quadratic convergence of semismooth Newton's method. We can use the globalized version of semismooth Newton’s method proposed in \cite{semi} to solve (\ref{convex}). We will take this method into account in Section \ref{sec4} to conduct our numerical comparison.
}

\section{Majorized Penalty Method for EDMFR}\label{sec3}
In this section, we will present the numerical algorithm for solving (\ref{EDMFR_former}), which is the majorized penalty method discussed in \cite{qi2018}.

Firstly, we can rewrite (\ref{EDMFR_former}) as the following compact form
\begin{equation}\label{zuizhong}
  \begin{split}
  \min_{D\in \mathcal{S}^{n+1}} &\ f(D):=\frac{1}{2}\|D -\Delta\|_F^2 \\
  \text{s.t.} &\ {-D\in\mathcal{K}^{n+1}_+(r)} \\
    &\ \mtB(D)=0,
\end{split}
\end{equation}
where $\mathcal{K}_{+}^{n+1}(r)$ is the conditional positive semidefinite cone with rank-$r$ cut defined by
\begin{equation}\label{Kr}
\mK_+^{n+1} (r):=\{D\in\mK_+^{n+1}: \ \rank({J_{n+1} DJ_{n+1}})\le r\}.
\end{equation}
Inspired by the majorization technique proposed by \cite{qi2018}, we could penalize $\mK_+^{n+1} (r)$, and then solve the majorized problem of the resulting problem. Such idea is also used to solve other types of nonconvex models, for example, EDM model with ordinal constraints \cite{edmoc}.
\subsection{{Penalizing} $\mathbf{\mathcal{K}^{n+1}_+(r)}$}\label{sec3.1}
{In order to introduce the majorized penalty method in \cite{qi2018}, we first give the following lemma which leads to the equivalent reformulation of the rank constraint $D\in\mathcal{K}^{n+1}_+(r)$.
\begin{lemma}{\rm\cite[{\it Lemma} 2.1]{qi2018}}\label{lemma1}
Given $D\in\mathcal{S}^{n+1}$ and integer $r\le n+1$. Let $\Pi^B_{\mK_+^{n+1}(r)}(D)$ denote the projection of $D$ onto nonconvex set $\mK_+^{n+1}(r)$. For any $\Pi_{\mK_+^{n+1}(r)}(D)\in\Pi^B_{\mK_+^{n+1}(r)}(D)$, we have the following results.
\begin{description}
  \item[{\rm(i)}]$\la \Pi_{\mK_+^{n+1}(r)}(D),D-\Pi_{\mK_+^{n+1}(r)}(D)\ra=0$.
  \item[{\rm(ii)}]The function
$$
h(D):=\frac12\|\Pi_{\mK_+^{n+1}(r)}(D)\|_F^2
$$
is convex and $\Pi_{\mK_+^{n+1}(r)}(D)\in\partial h(D)$, where $\partial h(D)$ is the subdifferential of $h$ at $D$.
  \item[{\rm(iii)}]One particular element of $\Pi^B_{\mK_+^{n+1}(r)}(D)$ {\rm(}denoted by $\Pi_{\mK_+^{n+1}(r)}(D)${\rm)} is given by
  $$
\Pi_{\mK_+^{n+1}(r)}(D)=\Pi_{\mathcal{S}_+^{n+1}(r)}({J_{n+1} DJ_{n+1}})+(D-{J_{n+1} DJ_{n+1}}).
$$
  For any $A\in\mathcal{S}^{n+1}_+$, $\Pi_{\mathcal{S}_+^{n+1}(r)}(A)$ can be calculated by
      $$
      \Pi_{\mathcal{S}_+^{n+1}(r)}(A)=\sum_{i=1}^r \max(0,\lambda_i)p_i p_i^{\top},
      $$
      with the spectral decomposition of $A$ given by
      $$
      A=\sum_{i=1}^n \lambda_i p_i p_i^{\top}.
      $$

\end{description}
\end{lemma}

}

The main idea of majorized penalty method \cite{qi2018} is to reformulate constraint $-D\in \mK_+^{n+1} (r)$ as the following equivalent form
\begin{equation}\label{g(A)=0}
-D\in \mK_+^{n+1} (r)\ \Longleftrightarrow\ g(D):=\frac{1}{2}\mathrm{dist}^2 (-D, \mK_+^{n+1} (r))=0.
\end{equation}
Here, $\mathrm{dist}(-D, \mK_+^{n+1} (r))$ denotes the distance between $-D$ and set $\mathcal{K}^{n+1}_+(r)$. By (\ref{g(A)=0}), problem (\ref{zuizhong}) is equivalent to
\begin{equation}\label{final-model}
  \begin{split}
  \min_{D\in \mS^{n+1}} &\ f(D) \\
  \text{s.t.} &\ g(D)=0 \\
  & \mtB(D)=0.
\end{split}
\end{equation}
Then we penalize $g(D)$ to the objective function and a majorization method is designed to solve the resulting penalty problem 
\begin{equation}
  \begin{split}
  \min_{D\in \mS^{n+1}} & \ {f_{\rho}(D):=f(D)+\rho g(D)}\\
  \text{s.t.} & \ D\in \Xi ,\\
\end{split}
\end{equation}
where $\rho>0$ is the penalty factor and $\Xi:=\{D\in \mS^{n+1}\mid\mtB(D)=0\}$.

{In order to use the majorized penalty method, we have the following result which gives the relationship between $g(\cdot)$ and $h(\cdot)$.
\begin{lemma}{\rm\cite[{\it Lemma} 2.2]{qi2018}}\label{lemma2.2}
The function $h(\cdot)$ in Lemma {\rm\ref{lemma1}} can be reformulated by
$$
h(D)=\frac12\|D\|_F^2-g(-D),
$$
where $g(D)$ is defined in {\rm(\ref{g(A)=0})}.
\end{lemma}
Based on the properties above, we need to obtain a majorization function of $g(D)$. Recall the definition of majorization function. A majorization function $g_m(D,A)$ of $g(D)$ at a given point $A\in\mathcal{S}^{n+1}$ satisfies the following conditions
$$
g_m (A,A)=g(A),\ g_m (D,A)\ge g(D),\ \forall D\in\mathcal{S}^{n+1}.
$$
By the convexity of $h(A)$ and $\Pi_{\mK_+^{n+1}(r)}(A)\in\partial h(A)$, we have
\be\label{partialh(D)}
h(D)-h(A)\ge \la \Pi_{\mK_+^{n+1}(r)}(A),D-A\ra, \ \forall D\in \mathcal{S}^{n+1}.
\ee
By (\ref{partialh(D)}), }the majorization function of $g(D)$ at a given point $A\in\mathcal{S}^{n+1}$ can be obtained by
\begin{equation}\label{fun-gm}
g_m (D,A):=\frac{1}{2}\|D\|_F^2 -h(-A)+\langle\Pi_{\mK^{n+1}_+ (r)}(-A), D-A\rangle.
\end{equation}
{It can be easily verified that $g_m (D,A)$ is a majorization function of $g(D)$.} Then at each iteration $D^k$, we solve the following subproblem
\begin{equation}\label{xk+1}
D^{k+1}=\arg\min_{D\in \Xi }f_{\rho}^m(D){:=f(D)+\rho g_m (D,D^k)}.
\end{equation}
{Majorization} function $f^m_{\rho}(D)$ in (\ref{xk+1}) can be simplified as
\begin{eqnarray}\label{x+}
  f_{\rho}^{m}(D,D^k)  & =&\frac{1}{2}\|D\|_F^2-\la D,\Delta\ra+\rho\left(\frac12 \|D\|_F^2+\la \Pi_{\mK_+^{n+1}(r)}(-D^k),D\ra\right) \nonumber\\
    & =&\frac12 \|D\|_F^2-\la D,\Delta^k\ra+\mathrm{const}\nonumber\\
    & =&\frac12 \|D-\Delta^k\|_F^2+\mathrm{const},
\end{eqnarray}
where $\Delta^k:=(\Delta-\rho\Pi_{\mK_+^{n+1}(r)}(-D^k))/(1+\rho)$ and 'const' represents the constant part with respect to $D$.
Therefore, the subproblem reduces to the following form
\begin{equation}\label{reduce-form}
\min_{D\in\Xi}\ \frac12 \|D-\Delta^k\|_F^2.
\end{equation}
\subsection{Solving Subproblem (\ref{reduce-form}) by Explicit Formula.}
{Subproblem} (\ref{reduce-form}) is the projection onto convex set $\Xi$ which admits explicit formula. We derive it below.

The Lagrangian function of problem (\ref{reduce-form}) is given by
$$
L(D,y)=\frac12\|D-\Delta^k\|_F^2-{\la \mathcal{B}(D),y\ra,}
$$
where $y\in\mathbb{R}^{n+2}$ is the Lagrange multiplier. We can obtain the KKT conditions as follows:
$$
\left\{\begin{array}{l}
         \nabla_D L(D,y)=D-\Delta^k-\mathcal{B}^*(y) =0\\
         \mathcal{B}(D)=0. \\
       \end{array}
\right.
$$
Due to convexity of subproblem (\ref{x+}), we can get the optimal solution by {solving the} KKT conditions:
\be\label{kkt}
D^{k+1}:=\Delta^k+\mathcal{B}^*(y),
\ee
where
\be\label{yk}
\left\{\begin{array}{l}
  y_{n+2}=-\frac{\la \Delta^k-\Diag(\diag(\Delta^k)),H\ra}{\la H-\Diag(\diag(H)),H\ra}, \\
   y_{1:n+1}=-\diag(\Delta^k + y_{n+2}^k H).
\end{array}\right.
\ee
In other words, subproblem (\ref{reduce-form}) admits explicit solution given by (\ref{kkt}) and (\ref{yk}). Now we are ready to give the framework of the majorization penalty method for (\ref{EDMFR_former}).
\begin{algorithm}
\caption{FRMPA: Majorized Penalty Approach for EDMFR}\label{alg1}
\begin{algorithmic}[1]
\REQUIRE Dissimilarity matrix $\Delta$, matrix $H$, penalty parameter $\rho>0$, dimension $r$.
\STATE \textbf{Initialize} $D^0 \in\mathcal{E}^{n+1}$ and $k:=0$.
\STATE \textbf{Calculate} $D^{k+1}$ by {\rm(\ref{kkt}) and (\ref{yk})}.
\STATE \textbf{Update} $k \leftarrow k + 1$ and go to {\rm Step 2} until convergence.
\end{algorithmic}
\end{algorithm}

Note that Algorithm \ref{alg1} mentioned above is specifically designed for SSLP, which solves EDMFR model (\ref{EDMFR_former}).

\section{Numerical Results}\label{sec4}
In this section, we will test {the semismooth Newton's method \cite{semi} for the convex relaxation model (\ref{convex}) (denoted as convex model with facial reduction technique, \verb"FRC") and} the proposed Algorithm \ref{alg1} (denoted by \verb"FRMPA") to {see the} performance. 
All the tests are {conducted} on a laptop in MATLAB R2016a with Intel(R) Core(TM) i5-6200 CPU @ 2.30GHz 2.40GHz, 4GB RAM.

For \verb"FRMPA", we set the termination condition:
$$
f_{prog}:=\frac{\mid f_{\rho} (D^k)-f_{\rho}(D^{k-1}) \mid }{1+f_{\rho}(D^{k-1})}<10^{-4}.
$$
In other words, when the objective function progresses relatively slowly, we believe that current iteration $D^k$ is a good iteration. In \verb"FRMPA", we set $\rho=0.1$.
{For \verb"FRC", since the convex problem (\ref{convex}) removes the constraint  $\rank(J_{n+1}DJ_{n+1})\le r$, the proportion of eigenvalues will be taken into account in the following comparative experiments. We define the proportion as
$$
\mathrm{Eigenratio}:=\sum_{i=1}^r \lambda_i(-J_{n+1}DJ_{n+1})/\sum_{i=1}^n \lambda_i(-J_{n+1}DJ_{n+1}),
$$
where $D$ is the final computed EDM. The resulting EDM is regarded good if $\mathrm{Eigenratio}\ge 90\%$.
}

After obtaining an EDM $D$ {by \verb"FRC" and \verb"FRMPA"}, we apply multidimensional scaling (\verb"cMDS") to obtain a set of data $\hat{\mathbf{x}}_1,...,\hat{\mathbf{x}}_n,\hat{\mathbf{x}}_{n+1}$. Together with the coordinates of sensors $\bfx_1,...,\bfx_n$, we conduct Procrustes process \cite{Procrustes} to rotate the sensors back to there original positions, meanwhile, we also obtain the estimated position of the source, denoted by {$\bar\bfx_{M}$}.

\subsection{Application in Single Source Localization Problem}
We compare \verb"FRC" and \verb"FRMPA" with the following methods, the Lagrangian dual method (\verb"LagD") proposed by \cite{qiLAG} for solving problem (\ref{nedm}), the facial reduction technique of SDP (\verb"FNEDM") for solving (\ref{sdpsslp}) \cite{psd}, the squared-range-based least squares (\verb"SR-LS") \cite{ex1}, {the standard} fixed point scheme (\verb"SFP"){, and the constrained }weighted least squares method (\verb"CWLS") \cite{Cheung}. 
We report cputime as well as the following measures to evaluate the {quality} of solutions: squared position error of method $M$ defined by
\be\label{err}
\mathrm{err}_M:=\|\mathbf{x}_{n+1} - {\bar\bfx_{M}}\|_2^2,
\ee
where ${\bar\bfx_{M}}$ is the estimated location provided by method $M$.

The following examples {are tested} for $r=2$ (E1$\sim$E4) and $r=3$ (E5$\sim$E6). Note that \verb"CWLS" is only for $r=2$.


\bit
\item[E1.] {\rm\cite[Example 1]{ex1}} 
In this example, we extend Example \ref{XH=0} {in Section \ref{sec2}} numerically. There is a Gaussian noise $\epsilon_i$ with mean $0$ and variance of $0.1$ between the target node and anchors $\mathbf{x}_i$, i.e., the noisy distance are $\|\mathbf{x}_{n+1}-\mathbf{x}_i\|_2+\epsilon_i$. The real and observed distances of target node $\mathbf{x}_{n+1}$ and anchor nodes $\mathbf{x}_i$ are
$$
\begin{array}{cccccc}
  \mathrm{exact} & \sqrt{65} & \sqrt{173} & \sqrt{85} & \sqrt{58} & \sqrt{61} \\
  \mathrm{noisy} & 8.0051 & 13.0112 & 9.1138 & 7.7924 & 8.0210
\end{array}
$$
respectively. The coordinate of the solution by \verb"FRMPA" {and \verb"FRC" are both} $(-1.9907, 3.0474)$, which is approximately to that of  \verb"FNEDM" and \verb"LagD" $(-1.9907, 3.0474)$, while the optimal solution obtained by \verb"SR-LS", \verb"SFP" and \verb"CWLS" are $(-2.0189, 2.9585)$, (-1.9916,3.0467) and (-1.9895,3.0431), respectively. We demonstrate them in {Fig.} \ref{fig:Ex1}.
\begin{figure}[H]
  \centering
  \includegraphics[width=0.72\textwidth]{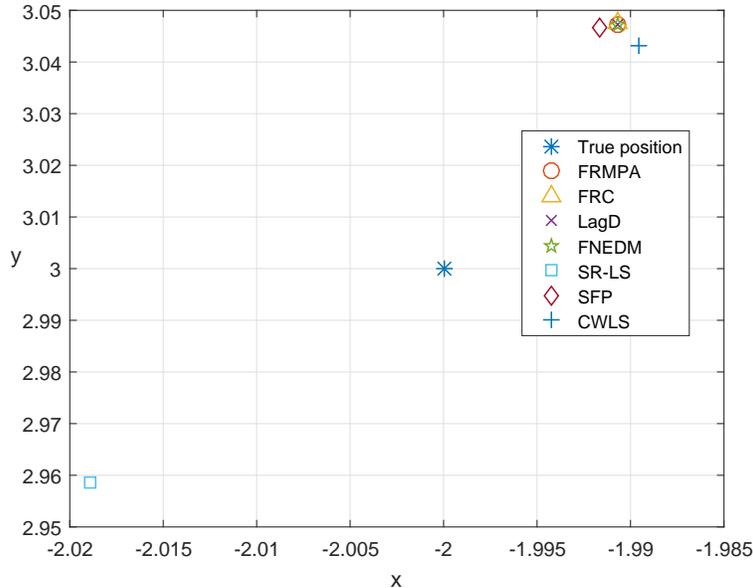}\\
  \caption{The positions solved by the seven methods in E1}\label{fig:Ex1}
\end{figure}
It can be seen from Table \ref{tab:Ex1} that \verb"CWLS" provides a solution with smallest squared error, while consuming the least cputime.
Comparing {four} solvers \verb"FRMPA", \verb"FRC", \verb"LagD" and \verb"FNEDM" for matrix models of SSLP, they return competitive solutions with almost the same accuracy, whereas our approach \verb"FRMPA" is the fastest among these {four} solvers.
\begin{table}[htbp]
  \centering
  \caption{Numerical results for E1}
     \begin{tabular}{ccc}
    \toprule
    Method & \multicolumn{1}{c}{err} & \multicolumn{1}{c}{time(s)} \\
    \midrule
    FRMPA & 2.33E-03 & 3.32E-03 \\
    FRC   & 2.34E-03 & 4.05E-03 \\
    LagD  & 2.33E-03 & 5.56E-02 \\
    FNEDM & 2.33E-03 & 4.32E-01 \\
    SR-LS & 2.08E-03 & 2.04E-03 \\
    SFP   & 2.32E-03 & 4.59E-04 \\
    CWLS  & 1.97E-03 & 4.89E-04 \\
    \bottomrule
    \end{tabular}%
  \label{tab:Ex1}%
\end{table}%

\item[E2.] \cite{Cheung} 
We follow Cheung et al. \cite{Cheung} {and} consider the sensors as the base stations and the source as the cellular phone. The five base stations are at coordinates $(0,0) m$, $(3000\sqrt{3}, 3000)m$, $(0, 6000)m$, $(-3000\sqrt{3}, 3000)m$, and $(-3000\sqrt{3}, -3000)m$. In \cite{Cheung}, the phone position was fixed at {$(1000, 2000)m$.} 

Consider the distance measurement model in (\ref{noisy}). The noises are normally distributed with mean zero and variance $\sigma$ between 90 and 180. All results are based on an average of 100 instances. As we can see in {Fig.} \ref{fig:Ex6}, the squared position error are very close among the {seven} methods. In terms of cputime, \verb"FRMPA", \verb"FRC", \verb"SR-LS", \verb"SFP" and \verb"CWLS" are all very fast, whereas \verb"LagD" and \verb"FNEDM" are not as fast as others.
\begin{figure}[H]
  \vspace{15pt}
  \centering
  \begin{minipage}{.48\linewidth}
\includegraphics[width=1\textwidth]{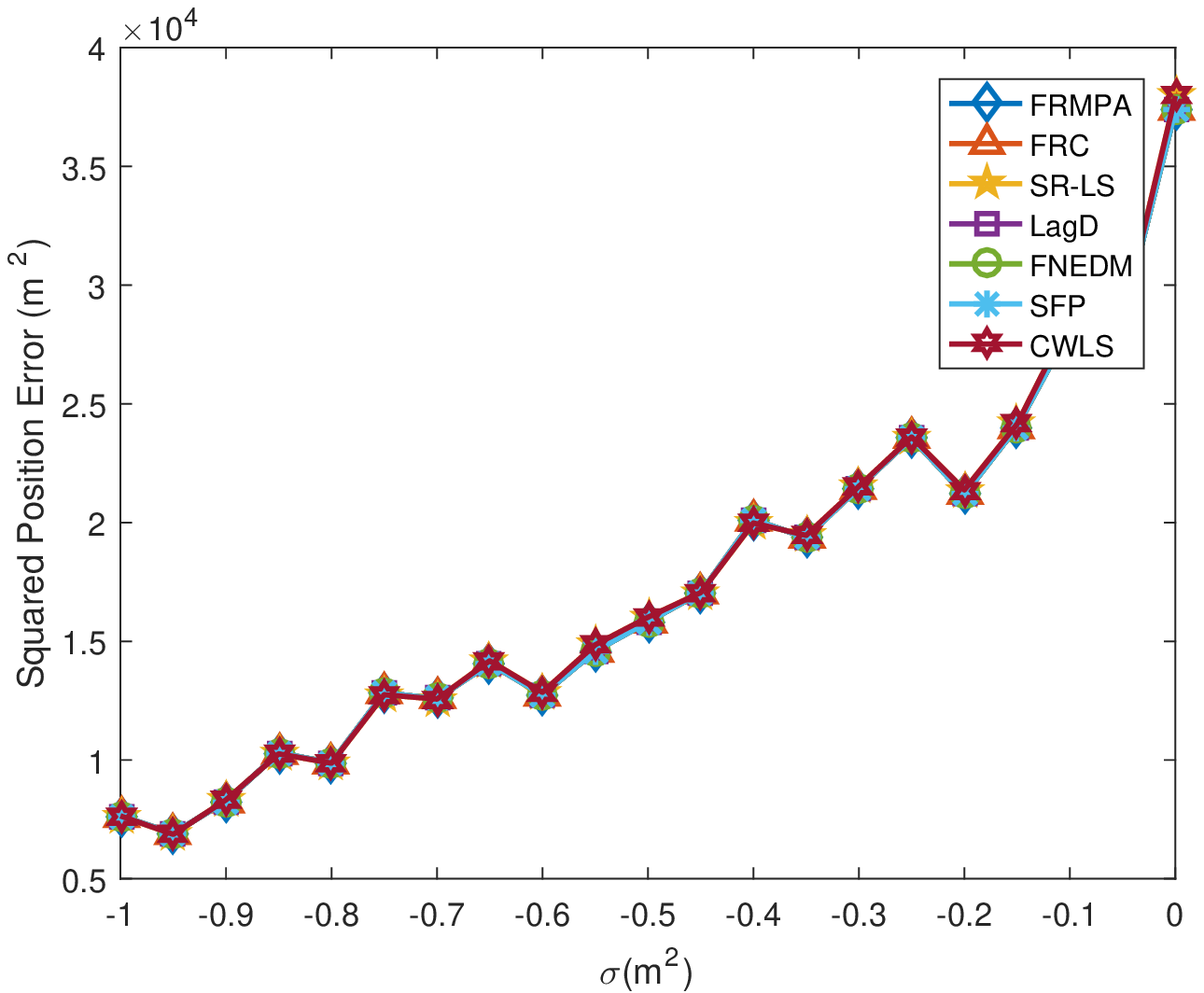}
  \caption*{(a)\ Squared position error}
  \end{minipage}
 \begin{minipage}{.48\linewidth}
\includegraphics[width=1\textwidth]{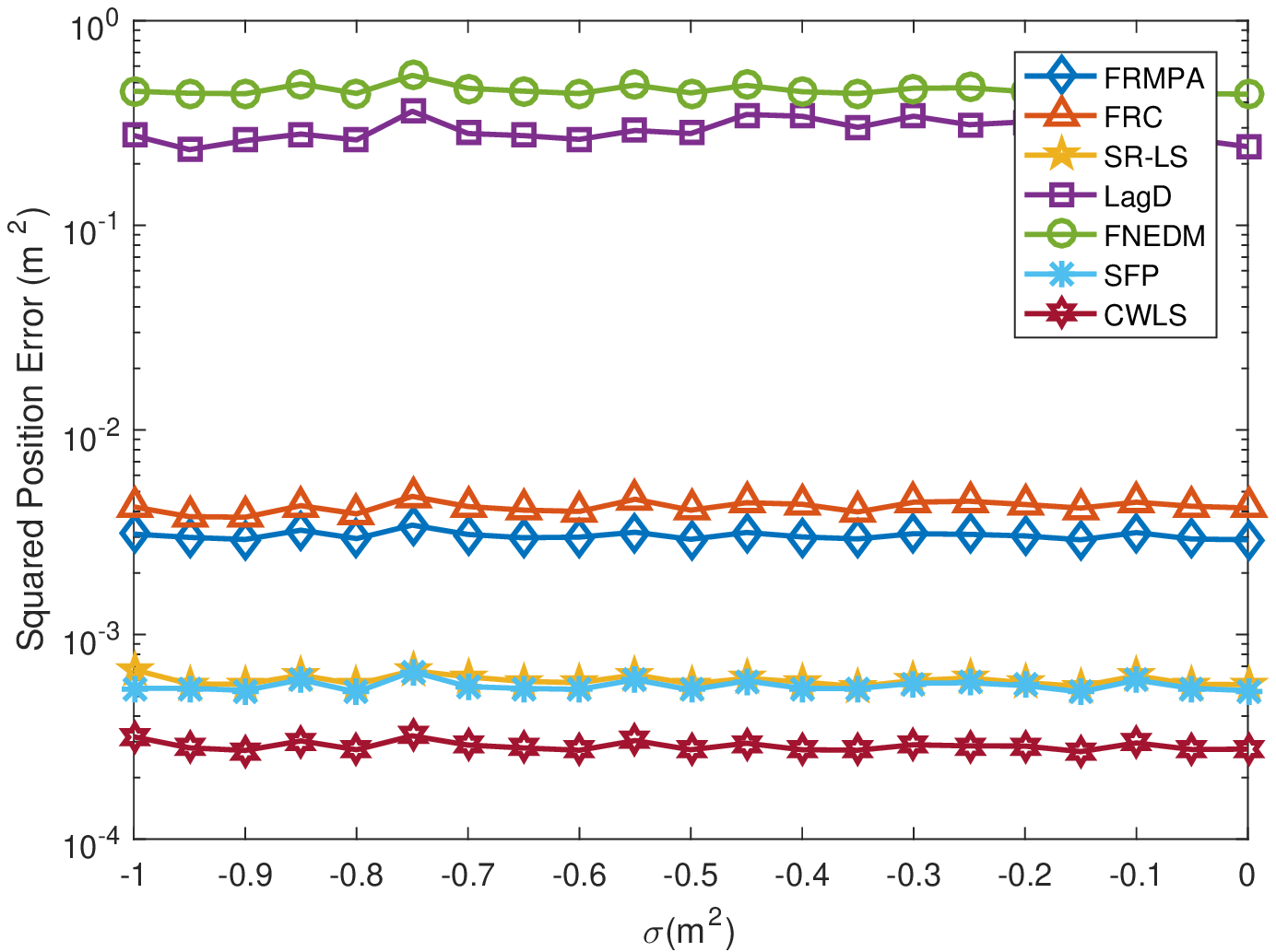}
  \caption*{(b)\ Computational time}
  \end{minipage}
 \caption{The 2-dimensional results in E2}\label{fig:Ex6}
 \vspace{0pt}
\end{figure}

\item[E3. ]{\rm\cite[Example 4.3]{ex2}}\ 
In this example, we randomly generate 100 instances. In each instance, there are five sensors whose positions are $\mathbf{x}_i$ and source position $\mathbf{x}_{n+1}$ are randomly generated from a uniform distribution on $[-10,10]\times [-10,10]$. In each instance, there is a normally distributed noise $\epsilon_i$ with mean zero and variance $\sigma ^2$. In other words, the observed distances between sensors and the source are
$$
\Delta_{n+1,i}=\Delta_{i,n+1}=\|\mathbf{x}_{n+1}-\mathbf{x}_i\|_2+\epsilon_i,\ i=1,...,n.
$$
The average results over 100 random instances are reported in Table \ref{tab:Ex2}.
It can be seen from {\rm Table \ref{tab:Ex2}} that for most cases the optimal solutions obtained by  {\verb"FRC",} \verb"FRMPA" and \verb"LagD" are very close and {they are} better than the other four methods. The method \verb"FNEDM" also performs well {in terms of the quality of solutions,} however, the cputime is not as fast as \verb"FRMPA". For the other three solvers \verb"SR-LS", \verb"SFP" and \verb"CWLS", they are very fast yet the quality of the solutions are not as good as \verb"FRMPA" and \verb"LagD".

\begin{table}[htbp]
  \centering
  \caption{Numerical results for E3}
    \begin{tabular}{cccccc}
    \toprule
    Method & $\sigma$ & 1.00E-03 & 1.00E-02 & 1.00E-01 & 1.00E+00 \\
    \midrule
    \multirow{2}[2]{*}{FRMPA} & err   & \textbf{1.04E-06} & \textbf{1.05E-04} & \textbf{1.05E-02} & \textbf{1.84E+00} \\
          & time(s) & 4.30E-03 & 4.28E-03 & 5.37E-03 & 5.19E-03 \\
    \midrule
    \multirow{2}[2]{*}{FRC} & err   & 1.06E-06 & \textbf{1.05E-04} & \textbf{1.05E-02} & \textbf{1.84E+00} \\
          & time(s) & 4.66E-03 & 4.57E-03 & 5.97E-03 & 7.83E-03 \\
    \midrule
    \multirow{2}[2]{*}{LagD} & err   & 1.07E-06 & \textbf{1.06E-04} & \textbf{1.05E-02} & 2.88E+00 \\
          & time(s) & 5.53E-02 & 5.38E-02 & 6.60E-02 & 1.36E-01 \\
    \midrule
    \multirow{2}[2]{*}{FNEDM} & err   & 1.05E-06 & \textbf{1.05E-04} & \textbf{1.05E-02} & 1.86E+00 \\
          & time(s) & 7.03E-03 & 8.33E-03 & 8.18E-03 & 8.73E-03 \\
    \midrule
    \multirow{2}[1]{*}{SR-LS} & err   & 1.71E-06 & 1.70E-04 & 1.69E-02 & 2.17E+00 \\
          & time(s) & 2.13E-03 & 1.90E-03 & 2.22E-03 & 2.20E-03 \\
    \midrule
    \multirow{2}[1]{*}{SFP} & err   & 2.38E+00 & 2.38E+00 & 2.41E+00 & 7.29E+00 \\
          & time(s) & 9.41E-04 & 8.34E-04 & 1.04E-03 & 1.12E-03 \\
    \midrule
    \multirow{2}[2]{*}{CWLS} & err   & 1.06E-06 & 1.06E-04 & \textbf{1.05E-02} & 1.87E+00 \\
          & time(s) & 4.62E-04 & 5.52E-04 & 5.13E-04 & 4.83E-04 \\
    \bottomrule
    \end{tabular}%
  \label{tab:Ex2}%
\end{table}%

\item[E4.] \cite[Example 4.2]{ex2} 
In this example, we generate 1000 implementations. In each implementation, {sensors $\mathbf{x}_j$ ($j=1,\cdots, n$)} and source $\mathbf{x}_{n+1}$ defined by {\rm(\ref{noisy})} are randomly generated by a uniform distribution on the square $ [-1000,1000]\times[-1000,1000]$. The observed distances $d_j$ between the sensors and the source are given by {\rm(\ref{noisy})} where there are normally distributed noises $\epsilon_j$ with mean zero and variance $20$.

The averaged results are shown in Table \ref{tab:Ex3}. We recorded the comparison between \verb"FRMPA" and {the other methods} when $n=4,5,8,10$. The third column is the number of runs out of 1000 in which the solution produced by the method was worse than the \verb"FRMPA" method in results of err defined in (\ref{err}). 
As we can see in Table \ref{tab:Ex3}, \verb"FRMPA" and \verb"CWLS" have competitive performance, and outperforms other solvers. The reason is explained below. Comparing \verb"FRMPA" with \verb"CWLS", they return similar averaged squared error, whereas \verb"FRMPA" has more beater runs than \verb"CWLS" over the 1000 random instances. However, \verb"CWLS" is a bit faster than \verb"FRMPA".
{Compared with \verb"FRMPA", \verb"FRC" does not seem to perform as well as \verb"FRMPA" because the rank constraint $\rank(-J_{n+1}DJ_{n+1})\le r$ is not considered in convex model problem (\ref{convex}). We analyzed the proportions of eigenvalues in the figure and verified our conjecture. As we can see in Fig. \ref{fig:E4ratio}, Eigenratio of \verb"FRMPA"  is very steady at 100\%, while Eigenratio of \verb"FRC" is not very stable. This explains that why there are some examples that \verb"FRC" gives larger estimated error than \verb"FRMPA" does.}

\begin{table}[htbp]
  \centering
  \caption{Comparison between FRMPA and the other methods for $n=4,5,8,10$ in E4}
    \begin{tabular}{ccccc}
    \toprule
    $n $    & Method $M$ & $\mathrm{err}_M>\mathrm{err}_{\mathrm{FRMPA}}$ & $\mathrm{err}_M$ & {time(s)} \\
    \midrule
    \multirow{7}[2]{*}{4} & FRMPA &   -    & 2.55E+03 & 4.29E-03 \\
          & FRC   & 528   & 2.73E+03 & 5.68E-03 \\
          & LagD  & 487   & 3.56E+03 & 1.13E-01 \\
          & FNEDM & 510   & 6.51E+04 & 4.20E-01 \\
          & SR-LS & 630   & 2.92E+03 & 2.07E+00 \\
          & SFP   & 514   & 7.97E+04 & 4.03E-01 \\
          & CWLS  & 500   & \textbf{2.49E+03} & \textbf{4.03E-04} \\
    \midrule
    \multirow{7}[2]{*}{5} & FRMPA &  -     & \textbf{6.26E+02} & 3.55E-03 \\
          & FRC   & 516   & \textbf{6.26E+02} & 4.97E-03 \\
          & LagD  & 512   & 6.22E+02 & 1.16E-01 \\
          & FNEDM & 492   & 4.68E+04 & 4.31E-01 \\
          & SR-LS & 663   & 9.58E+02 & 3.39E+00 \\
          & SFP   & 498   & 5.87E+04 & 2.84E-01 \\
          & CWLS  & 516   & 6.35E+02 & \textbf{2.84E-04} \\
    \midrule
    \multirow{7}[2]{*}{8} & FRMPA &   -    & \textbf{2.69E+02} & 2.91E-03 \\
          & FRC   & 509   & \textbf{2.69E+02} & 5.46E-03 \\
          & LagD  & 483   & \textbf{2.69E+02} & 1.60E-01 \\
          & FNEDM & 506   & 6.17E+03 & 5.36E-01 \\
          & SR-LS & 661   & 4.39E+02 & 8.08E-03 \\
          & SFP   & 525   & 9.88E+03 & 4.65E-01 \\
          & CWLS  & 507   & 2.72E+02 & \textbf{4.65E-04} \\
    \midrule
    \multirow{7}[2]{*}{10} & FRMPA &   -    & \textbf{2.16E+02} & 2.97E-03 \\
          & FRC   & 528   & \textbf{2.16E+02} & 5.19E-03 \\
          & LagD  & 502   & \textbf{2.16E+02} & 1.60E-01 \\
          & FNEDM & 511   & 4.88E+03 & 6.48E-01 \\
          & SR-LS & 652   & 3.32E+02 & 1.06E+01 \\
          & SFP   & 492   & 8.22E+03 & 4.56E-01 \\
          & CWLS  & 546   & 2.18E+02 & \textbf{4.56E-04} \\
    \bottomrule
    \end{tabular}%
  \label{tab:Ex3}%
\end{table}%

\begin{figure}[htbp]
  \vspace{15pt}
  \centering
  \begin{minipage}{.48\linewidth}
\includegraphics[width=1\textwidth]{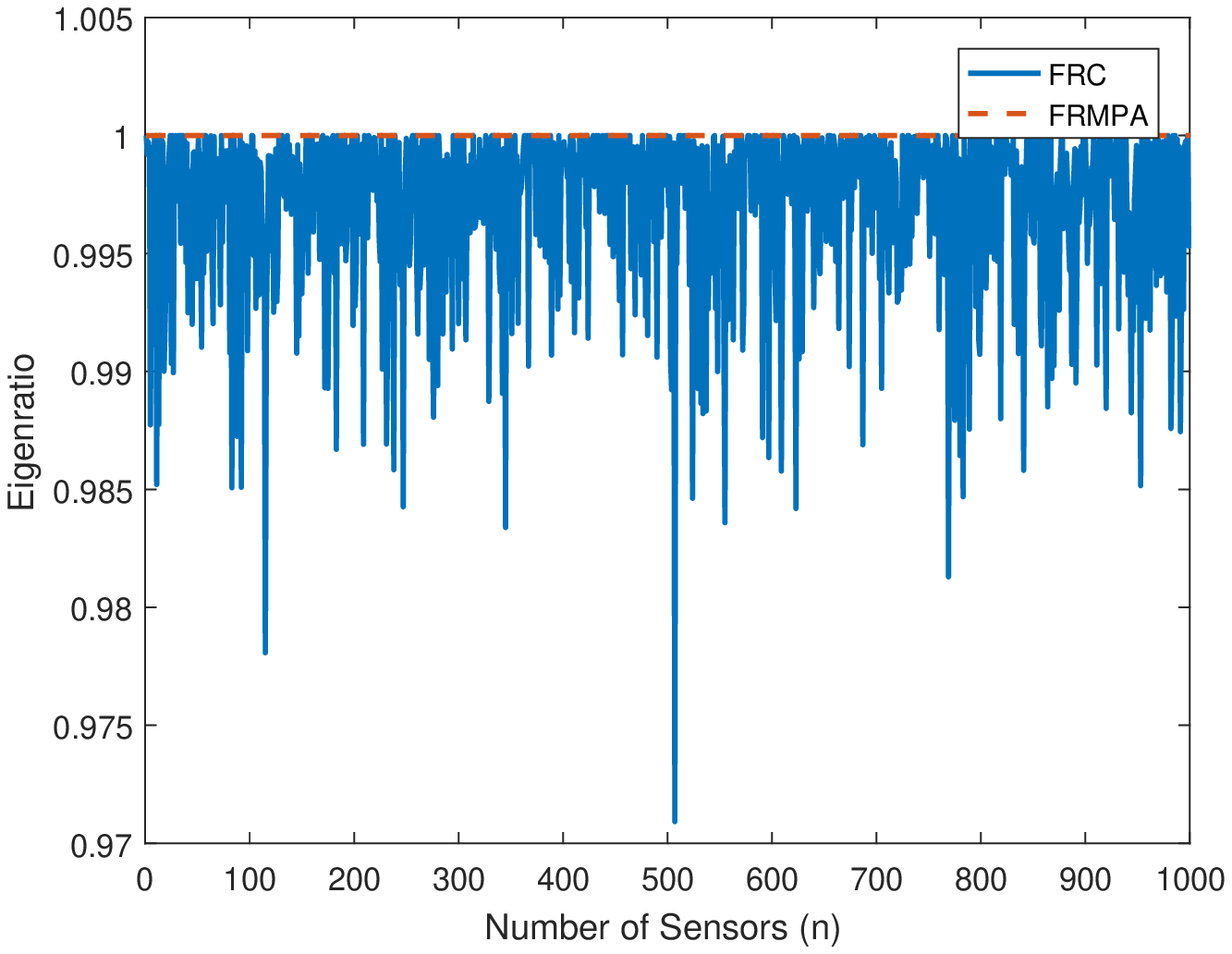}
  \caption*{(a)\ $n=4$}
  \end{minipage}
  \begin{minipage}{.48\linewidth}
\includegraphics[width=1\textwidth]{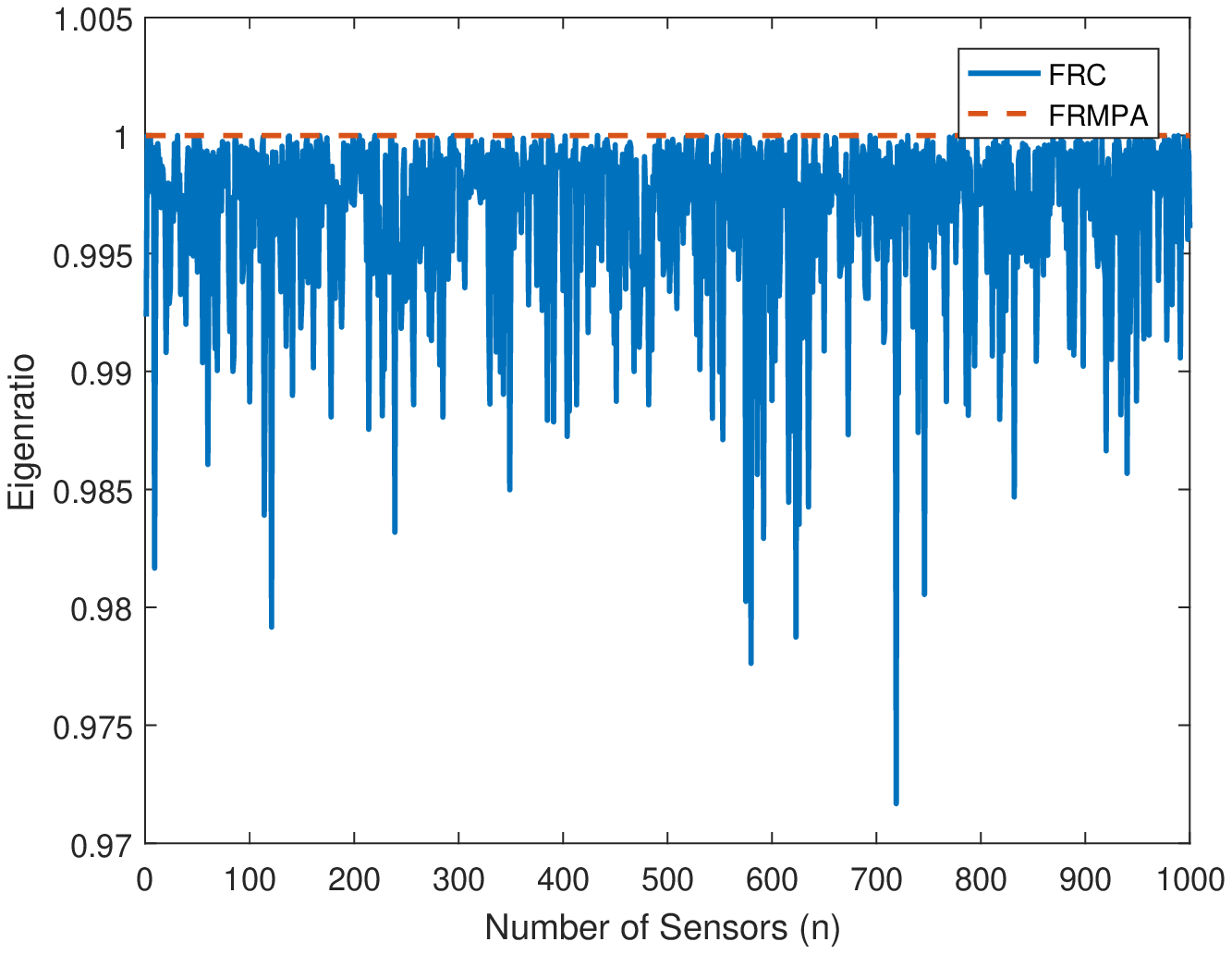}
  \caption*{(b)\ $n=5$}
  \end{minipage}
  \begin{minipage}{.48\linewidth}
\includegraphics[width=1\textwidth]{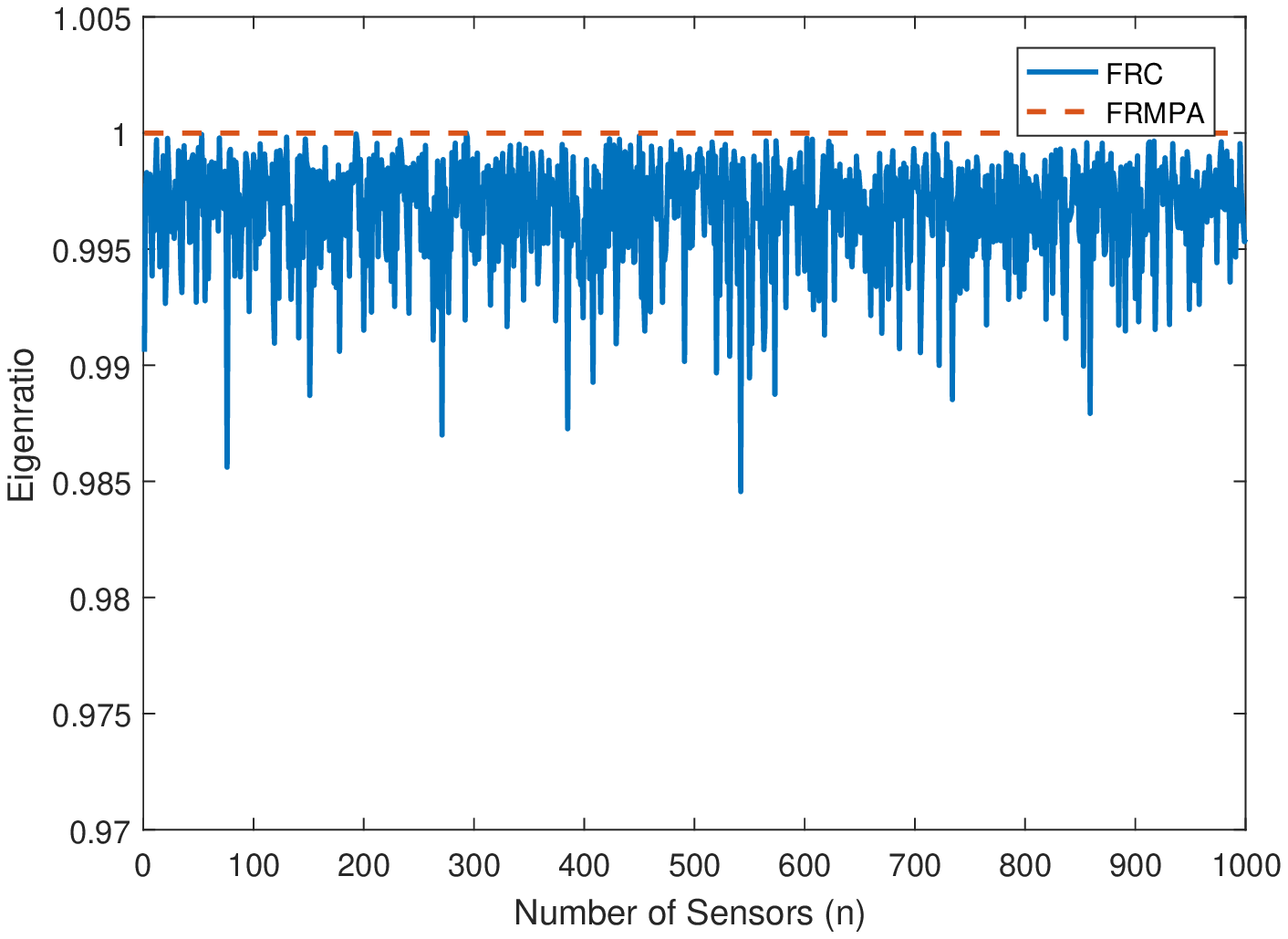}
  \caption*{(c)\ $n=8$}
  \end{minipage}
  \begin{minipage}{.48\linewidth}
\includegraphics[width=1\textwidth]{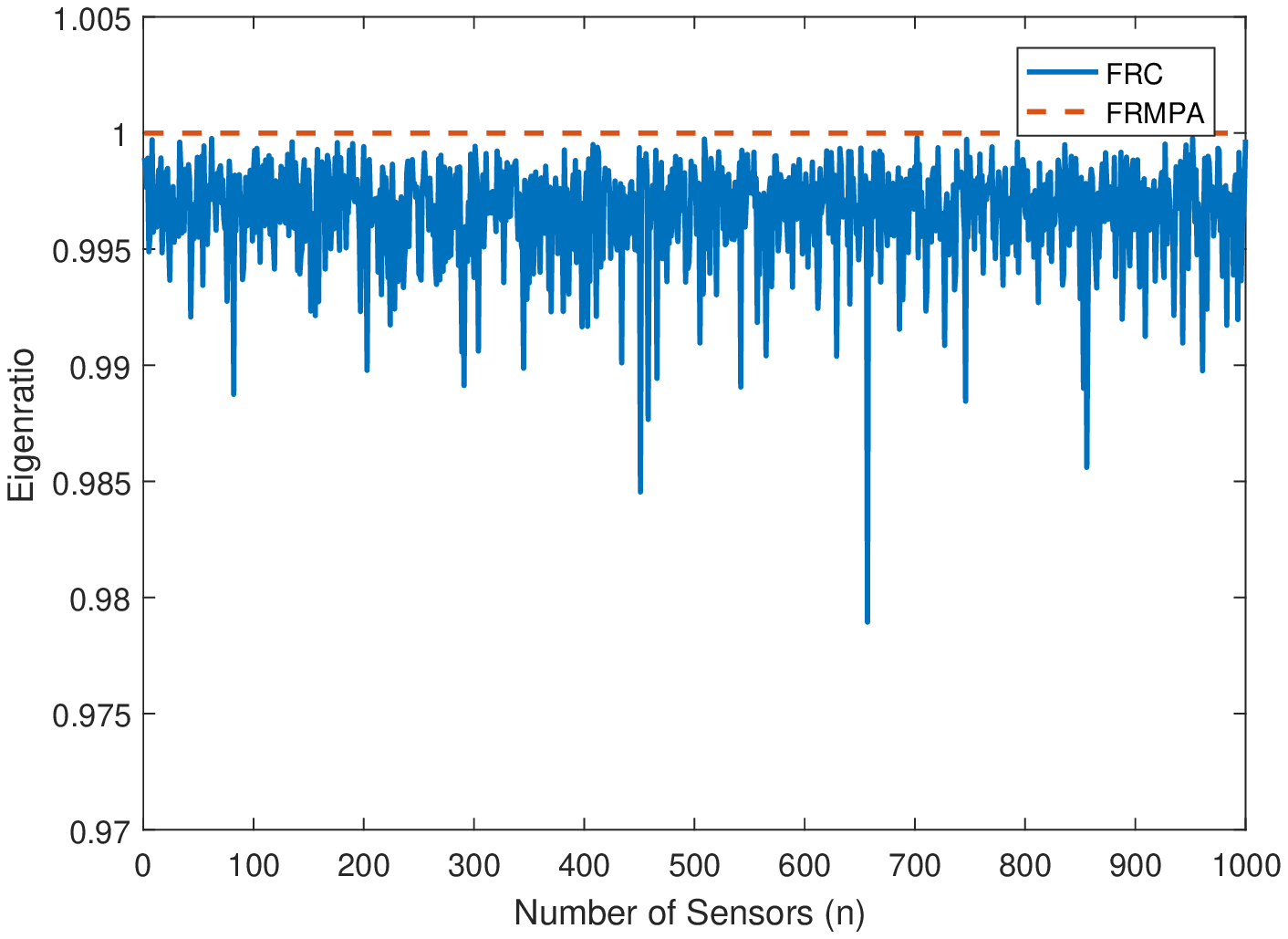}
  \caption*{(d)\ $n=10$}
  \end{minipage}
 \caption{The comparison results of Eigenratio for FRC and FRMPA in E4}\label{fig:E4ratio}
 \vspace{0pt}
\end{figure}
\eit

It should be noted that \verb"CWLS" is only restricted to the case of $r=2$. Therefore, below, we test some examples with $r=3$.

\bit
\item[E5.] \cite{psd} Similar to E2, the sensors $\bfx_j$ and {the} source $\bfx_{n+1}$ are randomly and uniformly distributed in $[-10,10]\times [-10,10]\times[-10,10]$, i.e., $r=3$. Unlike {\rm E2}, we randomly generate data with an noise that uniformly distributed, i.e., the observed distances between sensors and the source are
$$
\Delta_{n+1,i}=\Delta_{i,n+1}=\|\bfx_i-\bfx_{n+1}\|_2(1+\varepsilon_i),
$$
where $\varepsilon_i\in U(-\eta,\eta)$ is a uniformly distributed noise with noise factor $\eta$. {The} relative error $c^M_{re}$ between the true position of the source and the position obtained by method $M$ is given by
\be
c^M_{re}:=\frac{\|\bar\bfx_{M}-\bfx_{n+1}\|_2}{\|\bfx_{n+1}\|_2}.
\ee
We first test the special case of $\eta=0.2$. {Fig.} \ref{fig:Ex4-1} shows that the relative error decreases with the increase of $n$. {The solutions given by \verb"SR-LS" are not as good as the other methods.} For \verb"LagD" and \verb"FNEDM", it takes much more time, however, the other methods remain fast.

\begin{figure}[htbp]
  \vspace{15pt}
  \centering
  \begin{minipage}{.48\linewidth}
\includegraphics[width=1\textwidth]{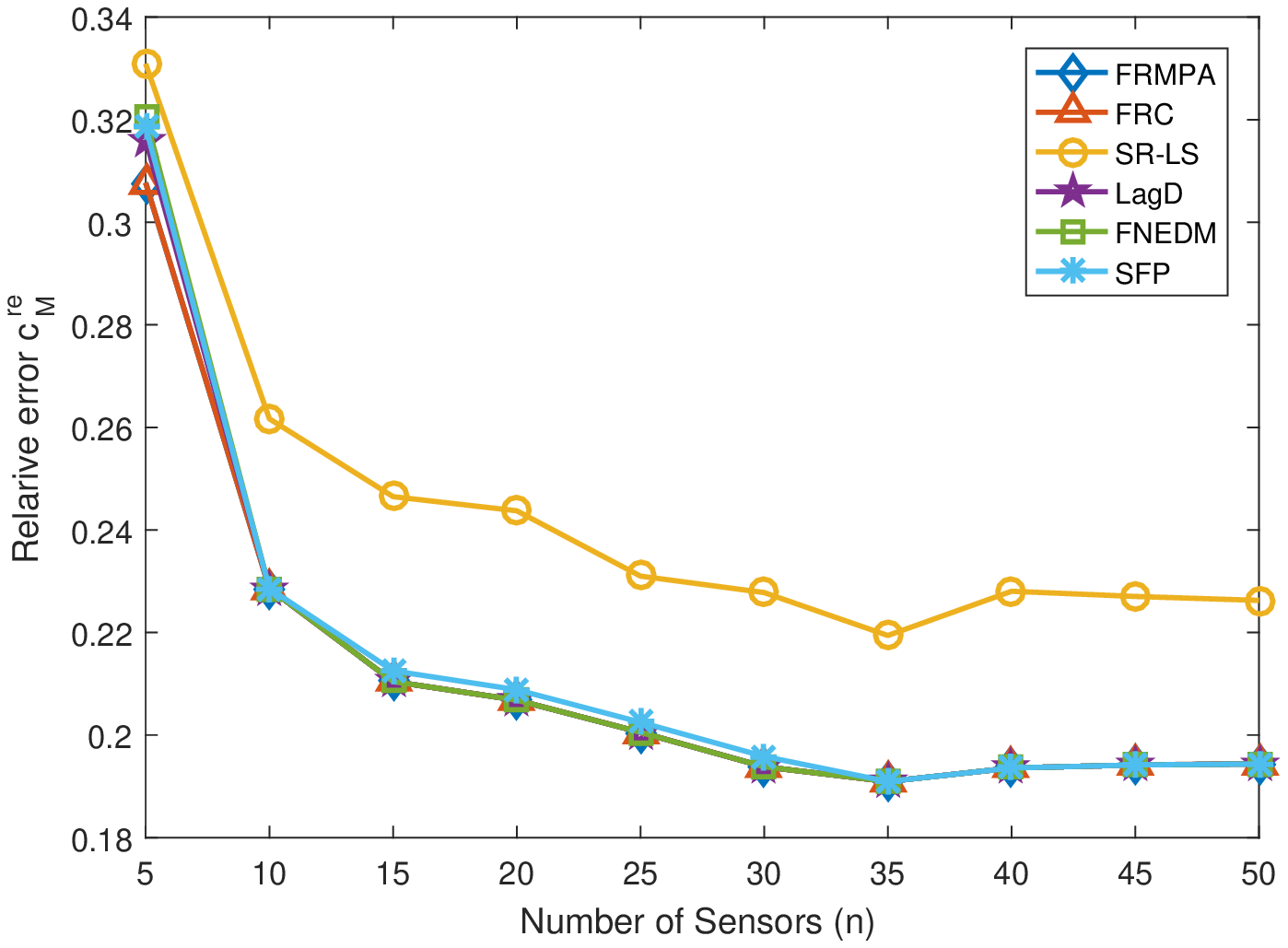}
  \caption*{(a)\ Relative error $c^M_{re}$}
  \end{minipage}
  \begin{minipage}{.48\linewidth}
\includegraphics[width=1\textwidth]{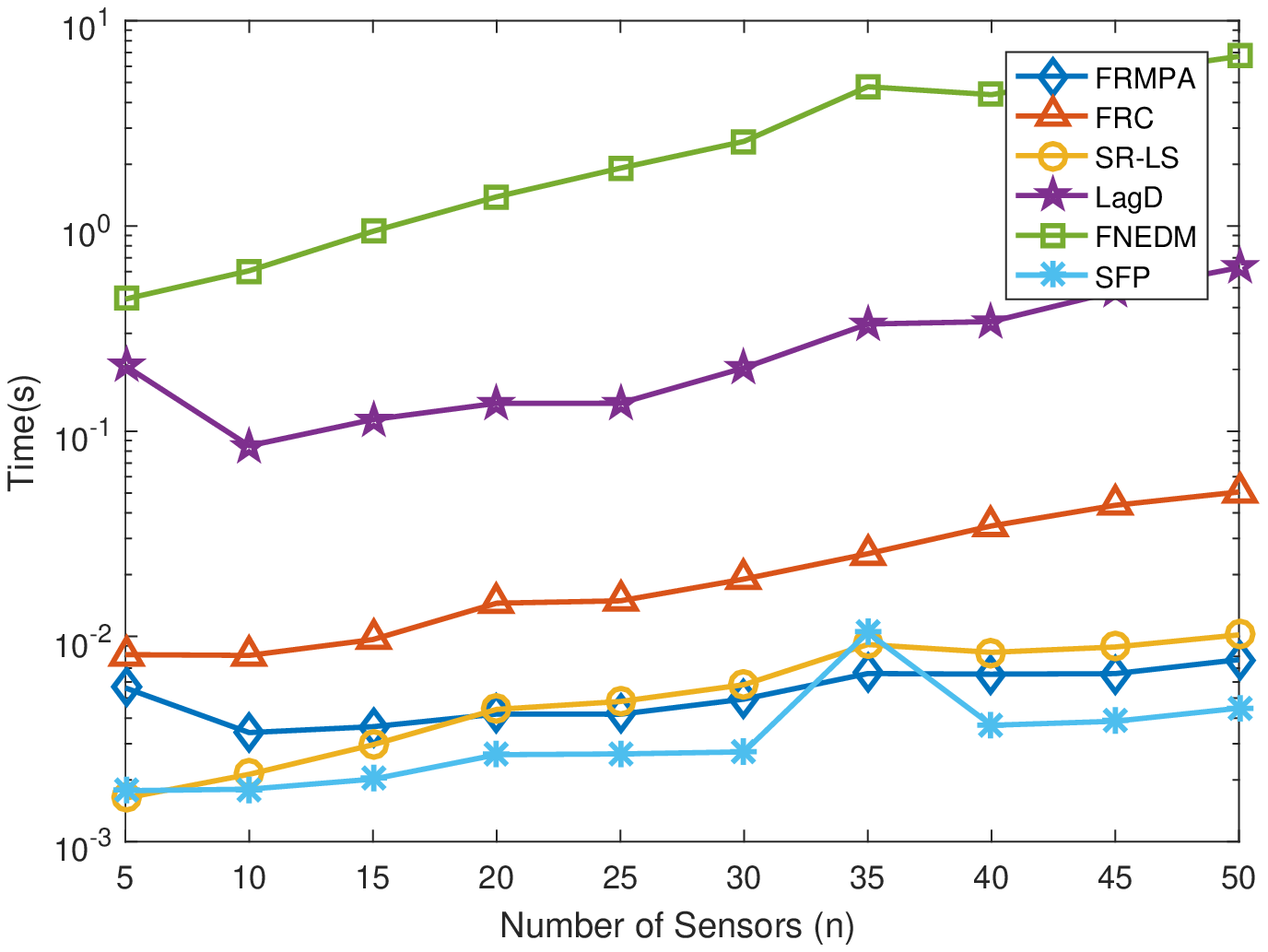}
  \caption*{(b)\ Computational time}
  \end{minipage}
 \caption{Numerical results for $\eta=0.2$ as $n$ changes in E5}\label{fig:Ex4-1}
 \vspace{0pt}
\end{figure}

Below we use performance profile to further evaluate the performance of each method. The performance profile is a plot that shows the general performance of all
the solvers. The $x$-axis represents the parameter $\tau$. It represents a ratio between a solver and the winner (such as relative error or cputime). That is, it describes its relative performance. The $y$-axis represents {the} probability $\psi_M(\tau)$ of problems for which each solver $M$ can get the best solver through $\tau$. For more details, see \cite{psd}.
The performance profiles can be seen in {Fig.} \ref{fig:Ex4}. The figure contains the performance profiles for the relative error $c_{re}^M$ and cputime.
In most of the cases, {Fig.} \ref{fig:Ex4}a shows that the {six} methods exhibit approximately good performance. \verb"FRMPA" obtains the best relative error for almost 100\% of the problem instances. {\verb"LagD" and \verb"FRC" is slightly worse than \verb"FRMPA".} The chances for the rest of the methods of winning are small, especially for \verb"SFP". Therefore, although \verb"SFP" wins in cputime as demonstrated in {Fig.} \ref{fig:Ex4}b compared {with} \verb"FRMPA", it has a poor performance in relative error.
{Fig.} \ref{fig:Ex4}b presents the performance profile for computational time. It seems that \verb"SR-LS" consumes less time. However, when $\tau $ increases, \verb"FRMPA" requires less time to get the final result. {\verb"FRC" is the fastest of the remaining three matrix optimization methods.} \verb"FNEDM" consumes {$10^{2}$ times or even $10^3$ times of cputime than that by \verb"FRMPA",} to achieve similar percentage of success. Compared with the other {three} EDM formulations \verb"FRMPA", {\verb"FRC"} and \verb"LagD", it is also clear that the SDP formulation consumes more time.

\begin{figure}[H]
  \vspace{15pt}
  \centering
  \begin{minipage}{.48\linewidth}
\includegraphics[width=1\textwidth]{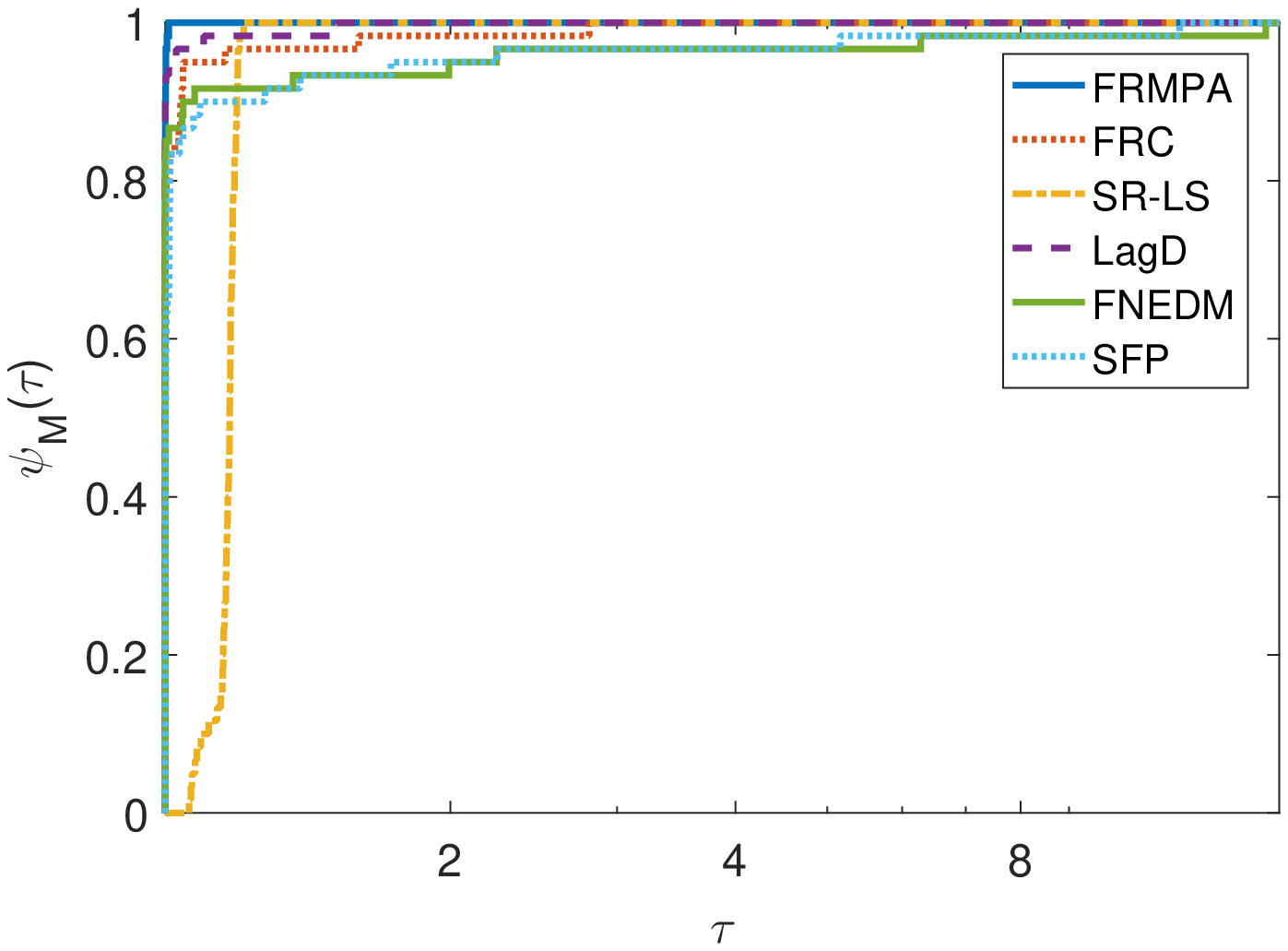}
  \caption*{(a)\ Relative error $c^M_{re}$}
  \end{minipage}
  \begin{minipage}{.48\linewidth}
\includegraphics[width=1\textwidth]{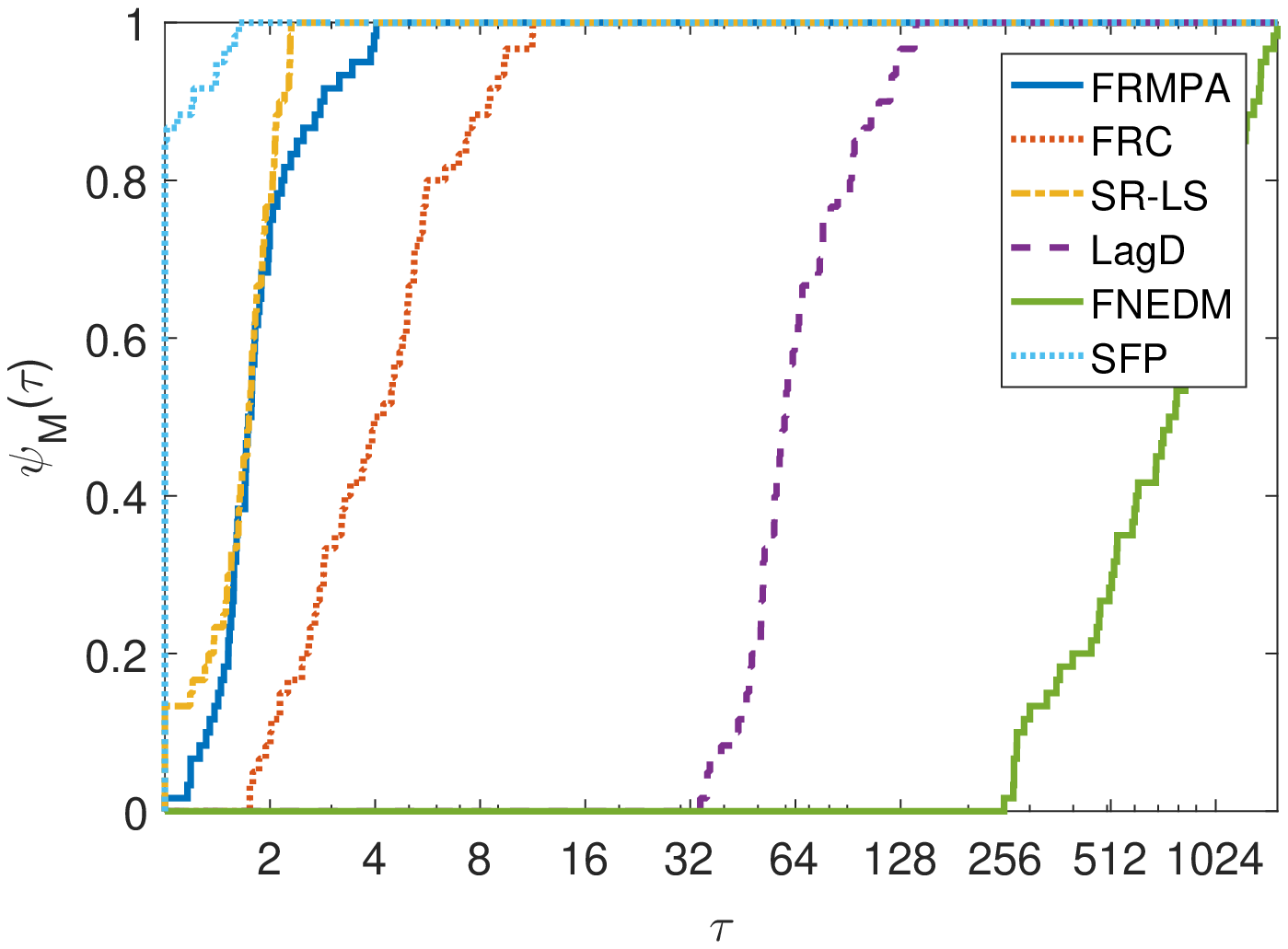}
  \caption*{(b)\ Computational time}
  \end{minipage}
 \caption{E5: performance profiles for the 100 random tests when $n=[5,10,15,20,25,30,35,40,45,50], \ \eta=[0.002,0.005,0.02,0.05,0.2,0.5]$}\label{fig:Ex4}
 \vspace{0pt}
\end{figure}
%


\item[E6. ] The data are tested in \cite[Example 2]{chenmeng}. In this example, we consider 3-dimensional case. Consider the distance measurement model in (\ref{noisy}). The sensors are $\bfx_1=(1,0,0),\ \bfx_2=(0,2,0), \ \bfx_3=(-2,-1,0),\ \bfx_4=(0,0,2),\ \bfx_5=(0,0,-1)$. The true source $\bfx_6$ is first at $(-1,1,1)$ (outside the convex hull of sensors) and then at $(0,0,0)$ (inside the convex hull). The noises $\epsilon_j$ are normally distributed with mean zero and variance between $10^{-1}$ and $10^{0}$. As we found that the 'err' among \verb"FRMPA", \verb"LagD" and \verb"FNEDM" are almost the same, we only compare \verb"FRMPA" {and \verb"FRC"} with the other two methods (\verb"SFP" and \verb"SR-LS"). As we can see in {Fig.} \ref{fig:Ex7}, \verb"FRMPA" {and \verb"FRC"} always {have} a good estimate whether the source in the convex hull or not. In other words, \verb"FRMPA" {and \verb"FRC" have more better} performance than the other two methods.
\begin{figure}[H]
  \vspace{15pt}
  \centering
  \begin{minipage}{.48\linewidth}
\includegraphics[width=1\textwidth]{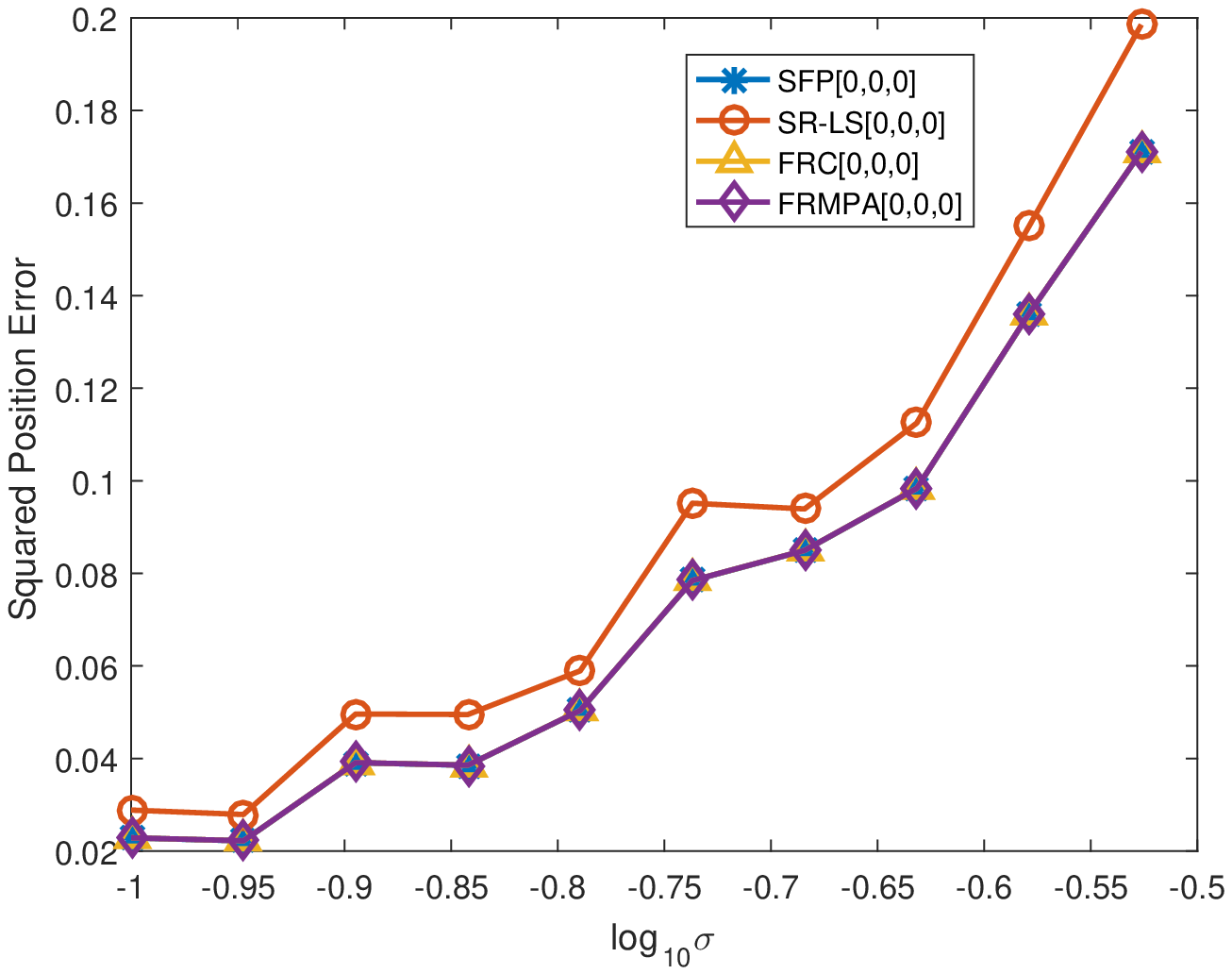}
  \caption*{(a)\ $\sigma\in[10^{-1},10^{-0.5}]$}
  \end{minipage}
  \begin{minipage}{.48\linewidth}
\includegraphics[width=1\textwidth]{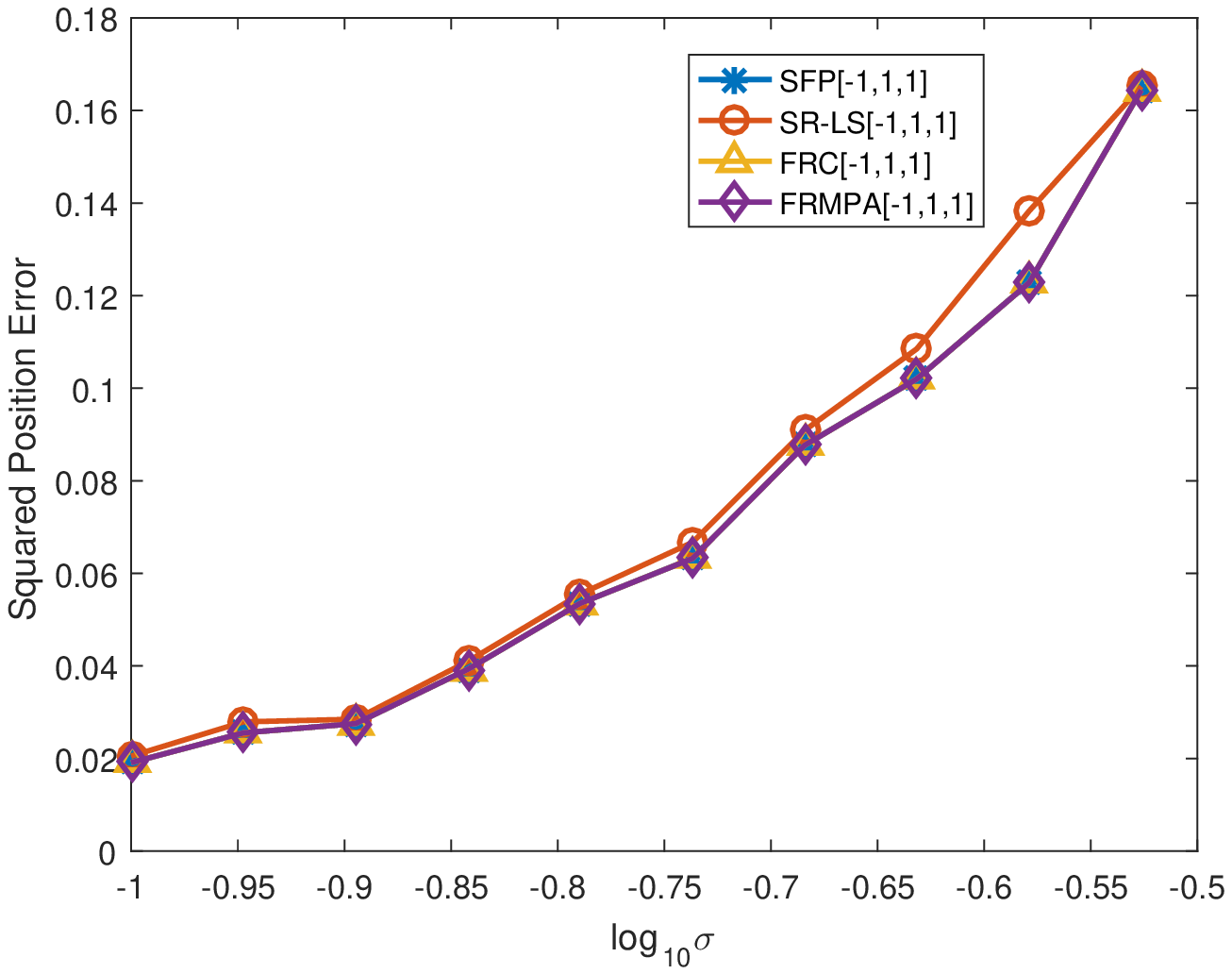}
  \caption*{(b)\ $\sigma\in[10^{-1},10^{-0.5}]$}
  \end{minipage}
  \begin{minipage}{.48\linewidth}
\includegraphics[width=1\textwidth]{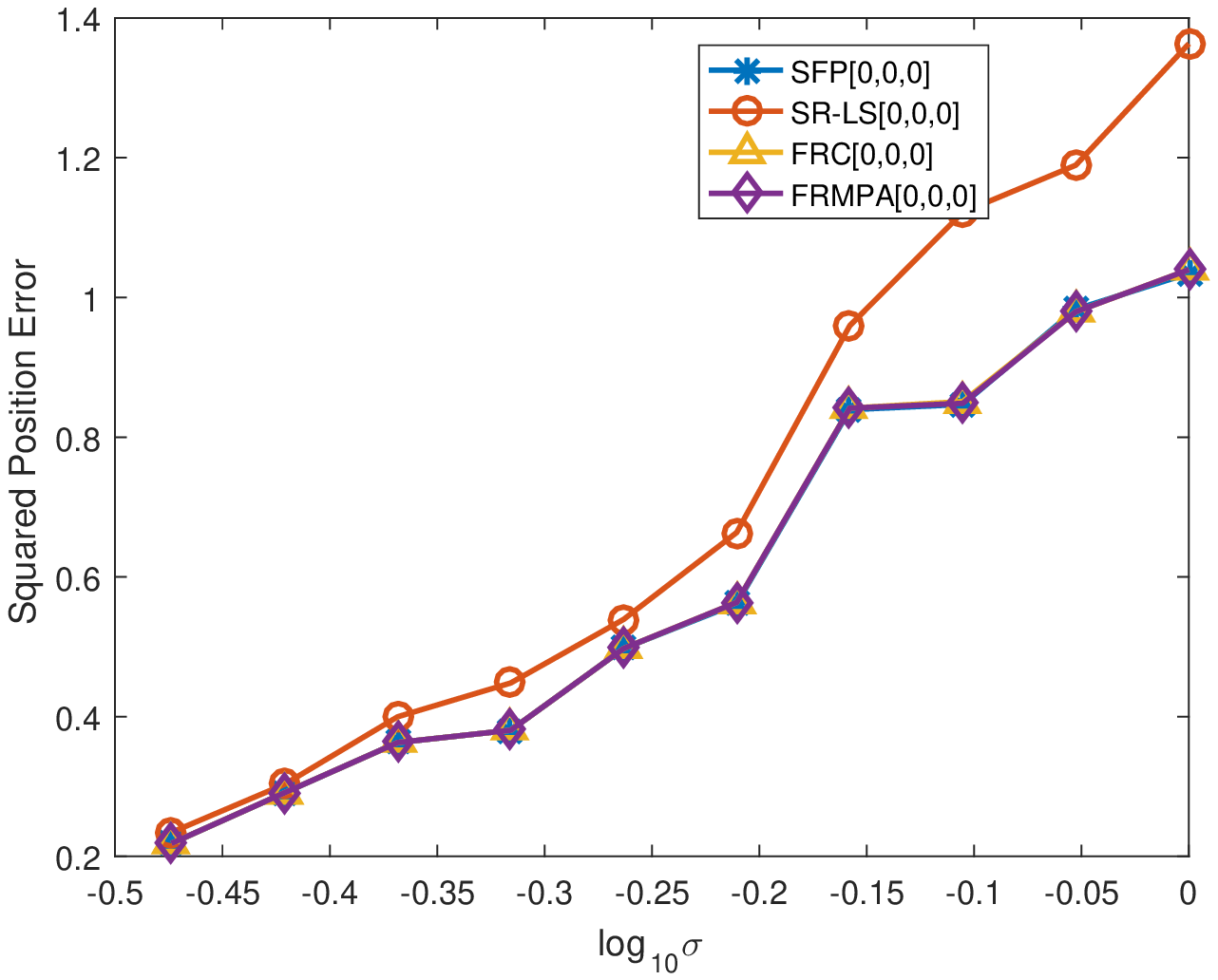}
  \caption*{(c)\ $\sigma\in[10^{-0.5},10^{0}]$}
  \end{minipage}
  \begin{minipage}{.48\linewidth}
\includegraphics[width=1\textwidth]{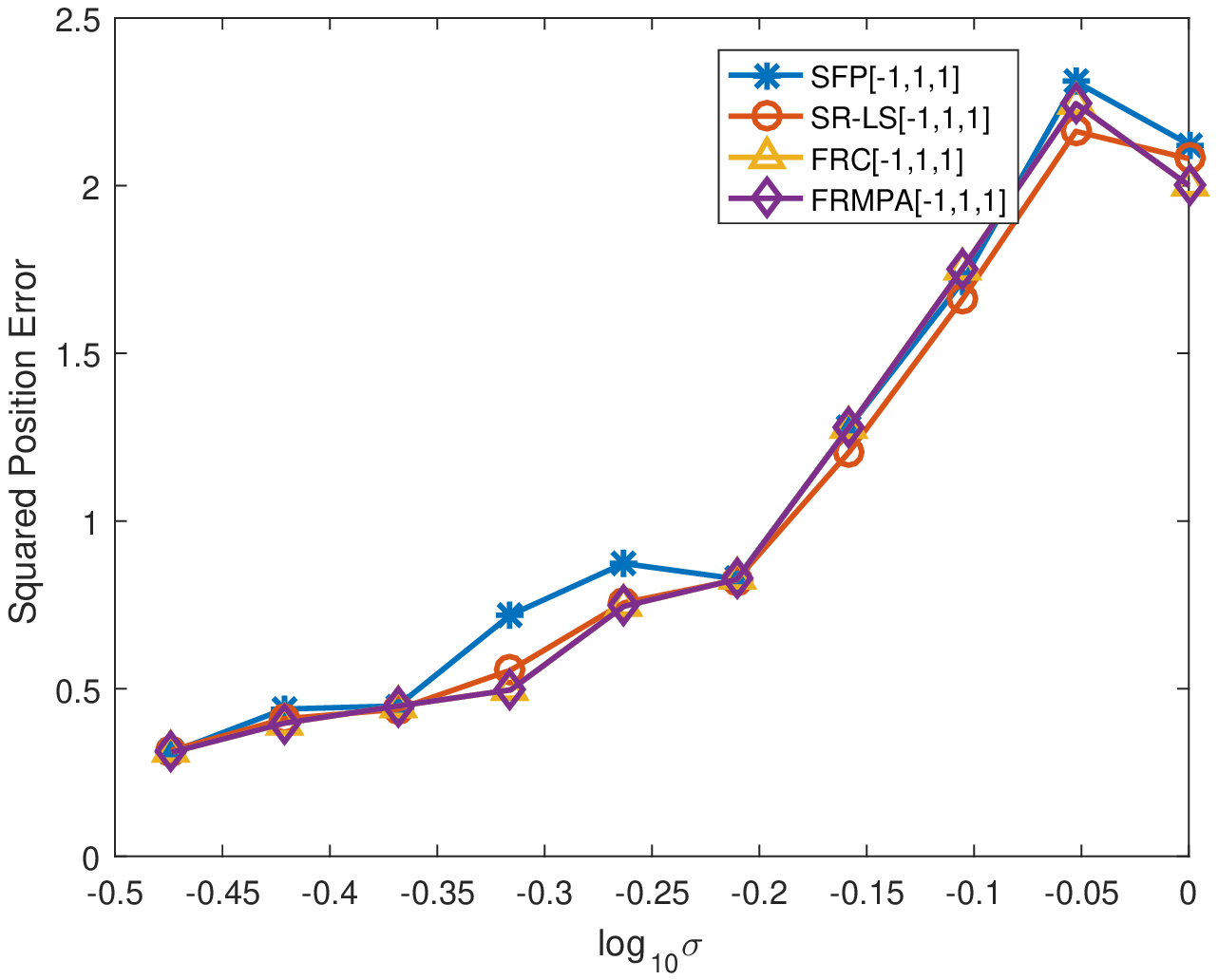}
  \caption*{(d)\ $\sigma\in[10^{-0.5},10^{0}]$}
  \end{minipage}
 \caption{The 3-dimensional error results in E6}\label{fig:Ex7}
 \vspace{0pt}
\end{figure}
\eit

{To summarize, \verb"FRMPA" and \verb"FRC" are competitive and enjoy advantages in terms of the solution quality and computational time over the rest methods. \verb"FRC" and \verb"FRMPA" sometimes perform similar, for example, E1, E2, E3 and E6. However, in E4 and E5, \verb"FRC" does not seem to perform as well as \verb"FRMPA". Comparing \verb"FRC" with \verb"FRMPA", \verb"FRMPA" is faster and more stable since it takes the rank constraint into account.}

\section{Conclusions}\label{sec5}
In this paper, we proposed a novel EDM model based on facial reduction. In theory, we derive the minimal face containing constraint (\ref{ft}) and express it as a closed formulation related to its exposing vector. We prove that constraint nondegeneracy is valid for every feasible point in the convex relaxation case. In terms of algorithm, we use the majorized penalty approach proposed by \cite{qi2018} whose subproblem admits closed form solution. The algorithm is simple and efficient. Numerical results show that the EDM model based on facial reduction performs well both in the quality of the solution and the speed.

\newpage
\noindent\textbf{Acknowledgments. \ }
{We would like to thank the editor for handling our paper, as well as two anonymous reviewers for their valuable comments, based on which the paper was further improved. We would also like to thank Professor Houduo Qi from the University of Southampton, UK, for his insightful suggestions on the derivation of $\face(\Omega,\mathcal{E}^{n+1})$.}


\end{document}